\documentclass[10pt, a4paper, twoside, fleqn]{article}
\usepackage{lmodern}

\newcommand{\PDFAuthor}{{O.~Mula, A.~Somacal}}
\newcommand{\PDFDocTitle}{A tutorial on Reduced Order Modeling}

\usepackage[printonlyused,withpage]{acronym}
\usepackage{amsmath}
\usepackage{amsthm}
\usepackage{amsfonts}
\usepackage{amssymb}

\usepackage{xcolor}

\definecolor{tue_red}{HTML}{C81919}

\definecolor{tue_red}{HTML}{C81919}
\definecolor{tue_dark_blue}{HTML}{101073}
\definecolor{tue_blue}{HTML}{0066CC}
\definecolor{tue_cyan}{HTML}{00A2DE}
\definecolor{tue_green}{HTML}{84D200}
\definecolor{tue_yellow}{HTML}{CEDF00}

\definecolor{accent}{gray}{0.95}

\definecolor{color1}{RGB}{0, 121, 178}
\definecolor{color2}{RGB}{255, 124, 37}
\definecolor{color3}{RGB}{37, 160, 55}
\definecolor{color4}{RGB}{220, 32, 44}
\definecolor{color5}{RGB}{147, 104, 186}
\definecolor{color6}{RGB}{143, 85, 76}
\definecolor{color7}{RGB}{230, 119, 192}
\definecolor{color8}{RGB}{127, 127, 127}
\definecolor{color9}{RGB}{192, 188, 55}
\definecolor{color10}{RGB}{0, 191, 206}

\usepackage[margin = 1in, twoside]{geometry}

\usepackage{microtype}

\usepackage[UKenglish]{babel}
\usepackage[UKenglish]{isodate}

\PassOptionsToPackage{hyphens}{url}
\usepackage[hypertexnames = false, pdftex, pdftitle = {\PDFDocTitle}, pdfauthor = {\PDFAuthor}]{hyperref}
\hypersetup{colorlinks = true, linkcolor = blue, citecolor = blue, urlcolor = blue, linktocpage}

\usepackage[capitalise]{cleveref}

\usepackage{natbib}
\setcitestyle{authoryear}
\usepackage{doi}

\usepackage{enumitem}

\usepackage{graphicx}
\usepackage{subcaption}

\usepackage{comment}

\usepackage{booktabs, tabu}

\newenvironment{absolutelynopagebreak}{}{}

\usepackage{algpseudocode}

\usepackage{multicol}

\providecommand{\keywords}[1]{\small\textbf{\textit{Keywords---}} #1}

\usepackage[T1]{fontenc}
\usepackage{bm}
\usepackage{utfsym}

\newcommand{\worstcase}{\ensuremath{\textsf{wc}}} 
\newcommand{\averagecase}{\ensuremath{\textsf{av}}} 

\newcommand{\Vn}{\ensuremath{{V_n}}}
\newcommand{\Wm}{\ensuremath{{W_m}}}

\newcommand{\dist}{\ensuremath{\text{dist}}}
\newcommand{\eps}{\ensuremath{\varepsilon}}
\newcommand{\opt}{\ensuremath{\textsf{opt}}}

\newcommand{\ord}{\ensuremath{\cO}}
\newcommand{\indicative}{\chi}

\renewcommand{\Im}{\ensuremath{\textsf{Im}}}
\newcommand{\poisson}{{\usym{1F420}}}
\newcommand{\cMpoisson}{\ensuremath{\cM^{\poisson}}}
\newcommand{\rad}{\ensuremath{\mathrm{rad}}}
\newcommand{\cen}{\ensuremath{\mathrm{cen}}}

\newcommand{\poly}{\ensuremath{\text{poly}}}
\newcommand{\SNN}{\ensuremath{\text{SNN}}}
\newcommand{\lin}{\ensuremath{\text{(lin)}}}

\newcommand{\pw}{\ensuremath{\text{(pw)}}}

\usepackage{mathtools}
\usepackage{suffix}
\usepackage{stmaryrd} %

\DeclarePairedDelimiter{\prt}{(}{)}

\DeclarePairedDelimiter{\abs}{\lvert}{\rvert}
\DeclarePairedDelimiter{\norm}{\lVert}{\rVert}
\DeclarePairedDelimiter{\inner}{\langle}{\rangle}
\DeclarePairedDelimiter{\set}{\{}{\}}

\let \oldforall \forall
\let \forall \undefined
\DeclareMathOperator{\forall}{\oldforall}

\let \oldexists \exists
\let \exists \undefined
\DeclareMathOperator{\exists}{\oldexists}

\let \oldtext \text
\renewcommand{\text}[1]{~\oldtext{#1}~}

\newcommand{\cond}{\, : \,}
\newcommand{\st}{\text{s.t.}}

\DeclareMathOperator*{\argmax}{arg \, max}
\DeclareMathOperator*{\argmin}{arg \, min}

\DeclareMathOperator*{\diag}{diag}

\usepackage{ifthen}
\newlength{\leftstackrelawd}
\newlength{\leftstackrelbwd}
\def \leftstackrel#1#2{\settowidth{\leftstackrelawd}%
  {${{}^{#1}}$} \settowidth{\leftstackrelbwd}{$#2$}%
  \addtolength{\leftstackrelawd}{- \leftstackrelbwd}%
  \leavevmode \ifthenelse{\lengthtest{\leftstackrelawd>0pt}}%
  {\kern-.5 \leftstackrelawd}{} \mathrel{\mathop{#2} \limits^{#1}}}

\DeclareMathOperator{\vspan}{span}

\newbool{isrelease}

\newcounter{review}

\newcommand{\corr}[1]{#1}
\newcommand{\new}[1]{#1}

\makeatletter

\newcommand \listreviewname{List of Reviews}
\newcommand \listofreviews{\section*{\listreviewname} \addcontentsline{toc}{section}{List of Reviews} \@starttoc{tor}}
\makeatother

\newcommand{\bN}{\ensuremath{\mathbb{N}}}

\newcommand{\bR}{\ensuremath{\mathbb{R}}}

\newcommand{\bU}{\ensuremath{\mathbb{U}}}

\newcommand{\cE}{\ensuremath{\mathcal{E}}}
\newcommand{\cF}{\ensuremath{\mathcal{F}}}

\newcommand{\cH}{\ensuremath{\mathcal{H}}}

\newcommand{\cK}{\ensuremath{\mathcal{K}}}
\newcommand{\cL}{\ensuremath{\mathcal{L}}}
\newcommand{\cM}{\ensuremath{\mathcal{M}}}

\newcommand{\cO}{\ensuremath{\mathcal{O}}}
\newcommand{\cP}{\ensuremath{\mathcal{P}}}

\newcommand{\cV}{\ensuremath{\mathcal{V}}}
\newcommand{\cW}{\ensuremath{\mathcal{W}}}

\newcommand{\rE}{\ensuremath{\mathrm{E}}}

\newcommand{\rU}{\ensuremath{\mathrm{U}}}

\newcommand{\tC}{\ensuremath{\mathtt{C}}}

\newcommand{\tG}{\ensuremath{\mathtt{G}}}

\newcommand{\tM}{\ensuremath{\mathtt{M}}}

\newcommand{\tP}{\ensuremath{\mathtt{P}}}

\newcommand{\R}{\bR}

\renewcommand{\d}{\mathrm{d}}
\newcommand{\dd}{\mathop{} \! \d}

\theoremstyle{definition}

\newtheorem{theorem}{Theorem}[section]
\newtheorem{corollary}[theorem]{Corollary}
\newtheorem{proposition}[theorem]{Proposition}
\newtheorem{lemma}[theorem]{Lemma}
\newtheorem{definition}[theorem]{Definition}
\newtheorem{problem}[theorem]{Problem}
\newtheorem{algorithm}[theorem]{Algorithm}

\newtheorem{remark}[theorem]{Remark}

\usepackage[framemethod = TikZ]{mdframed}

\renewenvironment{theorem}[1][]{
  \refstepcounter{theorem}
  \ifstrempty{#1}{
    \mdfsetup{
      frametitle = {
          \tikz[baseline = (current bounding box.east), outer sep = 0pt]
          \node[anchor = east, rectangle, fill = tue_red!20, text = black]
          {\strut Theorem~\thetheorem};
        }
    }
  }{
    \mdfsetup{
      frametitle = {
          \tikz[baseline = (current bounding box.east), outer sep = 0pt]
          \node[anchor = east, rectangle, fill = tue_red!20, text = black]
          {\strut Theorem~\thetheorem:~#1};
        }
    }
  }
  \mdfsetup{
    innertopmargin = 10pt,
    linecolor = tue_red,
    linewidth = 2pt,
    topline = true,
    leftline = false,
    rightline = false,
    bottomline = true,
    frametitleaboveskip = -1em,
  }
  \begin{absolutelynopagebreak}
    \vspace{1.0em}
    \begin{mdframed}[] \relax \vspace{-0.5em}
      }{
    \end{mdframed}
  \end{absolutelynopagebreak}
}
\crefname{theorem}{Theorem}{Theorems}

\newcounter{proposition}
\counterwithin{proposition}{section}
\renewenvironment{proposition}[1][]{
  \setcounter{proposition}{\value{theorem}}
  \refstepcounter{proposition}
  \setcounter{theorem}{\value{proposition}}
  \ifstrempty{#1}{
    \mdfsetup{
      frametitle = {
          \tikz[baseline = (current bounding box.east), outer sep = 0pt]
          \node[anchor = east, rectangle, fill = tue_dark_blue!20, text = black]
          {\strut Proposition~\thetheorem};
        }
    }
  }{
    \mdfsetup{
      frametitle = {
          \tikz[baseline = (current bounding box.east), outer sep = 0pt]
          \node[anchor = east, rectangle, fill = tue_dark_blue!20, text = black]
          {\strut Proposition~\thetheorem:~#1};
        }
    }
  }
  \mdfsetup{
    innertopmargin = 10pt,
    linecolor = tue_dark_blue,
    linewidth = 2pt,
    topline = true,
    leftline = false,
    rightline = false,
    bottomline = true,
    frametitleaboveskip = -1em,
  }
  \begin{absolutelynopagebreak}
    \vspace{1.0em}
    \begin{mdframed}[] \relax \vspace{-0.5em}
      }{
    \end{mdframed}
  \end{absolutelynopagebreak}
}
\crefname{proposition}{Proposition}{Propositions}

\newcounter{lemma}
\counterwithin{lemma}{section}
\renewenvironment{lemma}[1][]{
  \setcounter{lemma}{\value{theorem}}
  \refstepcounter{lemma}
  \setcounter{theorem}{\value{lemma}}
  \ifstrempty{#1}{
    \mdfsetup{
      frametitle = {
          \tikz[baseline = (current bounding box.east), outer sep = 0pt]
          \node[anchor = east, rectangle, fill = tue_cyan!20, text = black]
          {\strut Lemma~\thetheorem};
        }
    }
  }{
    \mdfsetup{
      frametitle = {
          \tikz[baseline = (current bounding box.east), outer sep = 0pt]
          \node[anchor = east, rectangle, fill = tue_cyan!20, text = black]
          {\strut Lemma~\thetheorem:~#1};
        }
    }
  }
  \mdfsetup{
    innertopmargin = 10pt,
    linecolor = tue_cyan,
    linewidth = 2pt,
    topline = true,
    leftline = false,
    rightline = false,
    bottomline = true,
    frametitleaboveskip = -1em,
  }
  \begin{absolutelynopagebreak}
    \vspace{1.0em}
    \begin{mdframed}[] \relax \vspace{-0.5em}
      }{
    \end{mdframed}
  \end{absolutelynopagebreak}
}
\crefname{lemma}{Lemma}{Lemmata}

\newcounter{definition}
\counterwithin{definition}{section}
\renewenvironment{definition}[1][]{
  \setcounter{definition}{\value{theorem}}
  \refstepcounter{definition}
  \setcounter{theorem}{\value{definition}}
  \ifstrempty{#1}{
    \mdfsetup{
      frametitle = {
          \tikz[baseline = (current bounding box.east), outer sep = 0pt]
          \node[anchor = east, rectangle, fill = tue_green!20, text = black]
          {\strut Definition~\thetheorem};
        }
    }
  }{
    \mdfsetup{
      frametitle = {
          \tikz[baseline = (current bounding box.east), outer sep = 0pt]
          \node[anchor = east, rectangle, fill = tue_green!20, text = black]
          {\strut Definition~\thetheorem:~#1};
        }
    }
  }
  \mdfsetup{
    innertopmargin = 10pt,
    linecolor = tue_green,
    linewidth = 2pt,
    topline = true,
    leftline = false,
    rightline = false,
    bottomline = true,
    frametitleaboveskip = -1em,
  }

  \begin{absolutelynopagebreak}
    \vspace{1.0em}
    \begin{mdframed}[] \relax \vspace{-0.5em}
      }{
    \end{mdframed}
  \end{absolutelynopagebreak}
}
\crefname{definition}{Definition}{Definitions}

\newcounter{problem}
\counterwithin{problem}{section}

\crefname{problem}{Problem}{Problems}

\newcounter{algorithm}
\counterwithin{algorithm}{section}

\crefname{algorithm}{Algorithm}{Algorithms}

\newcounter{corollary}
\counterwithin{corollary}{section}
\renewenvironment{corollary}[1][]{
  \setcounter{corollary}{\value{theorem}}
  \refstepcounter{corollary}
  \setcounter{theorem}{\value{corollary}}
  \ifstrempty{#1}{
    \mdfsetup{
      frametitle = {
          \tikz[baseline = (current bounding box.east), outer sep = 0pt]
          \node[anchor = east, rectangle, fill = tue_dark_blue!20, text = black]
          {\strut Corollary~\thetheorem};
        }
    }
  }{
    \mdfsetup{
      frametitle = {
          \tikz[baseline = (current bounding box.east), outer sep = 0pt]
          \node[anchor = east, rectangle, fill = tue_dark_blue!20, text = black]
          {\strut Corollary~\thetheorem:~#1};
        }
    }
  }
  \mdfsetup{
    innertopmargin = 10pt,
    linecolor = tue_dark_blue,
    linewidth = 2pt,
    topline = true,
    leftline = false,
    rightline = false,
    bottomline = true,
    frametitleaboveskip = -1em,
  }

  \begin{absolutelynopagebreak}
    \vspace{1.0em}
    \begin{mdframed}[] \relax \vspace{-0.5em}
      }{
    \end{mdframed}
  \end{absolutelynopagebreak}
}
\crefname{corollary}{Corollary}{Corollaries}

\usepackage{tikz}

\usetikzlibrary{backgrounds}

\usepackage{pgfplots}
\pgfplotsset{compat = 1.18}

\usepgfplotslibrary{groupplots}

\usepackage{pgfplotstable}

\booltrue{isrelease}

\definecolor{refkey}{rgb}{.8, .8, .8}
\definecolor{labelkey}{rgb}{.8, .8, .8}

\makeatletter
\newenvironment{align+}{%
  \start@align \@ne \st@rredfalse \m@ne
}{%
  \endalign
}
\makeatother

\begin{document}

    \title{
        Reduced Order Modeling and Applications to Inverse State Estimation\footnote{Chapter of the book \textit{An Introduction to Scientific Machine Learning: Mathematical Foundations and Numerical Implementation}}
    }

    \author{Olga Mula and Agustín Somacal}
    \maketitle

    \begin{abstract}
        Solving parametric partial differential equations (PDEs) repeatedly for
        many parameter values, as needed in optimal control, uncertainty
        quantification, and inverse problems, is often prohibitive with
        classical discretizations.
        This chapter introduces Reduced Order
        Modeling (ROM), which builds compressed yet accurate representations of
        parametric solution sets for fast online evaluation.
        We formulate the
        approximation of parametric PDEs as a supervised learning task and
        review linear approximation, very effective for elliptic and parabolic
        problems, before turning to nonlinear methods needed when solutions
        exhibit discontinuities or steep gradients.
        We then show how to use
        reduced-order models to efficiently recover the state of a physical
        system from limited measurements, which is an inverse problem known as state estimation.
        Throughout, we favor results with theoretical
        guarantees while also discussing practical methods, and provide a
        companion Jupyter notebook implementing the main algorithms.
    \end{abstract}


    \tableofcontents

    \pagestyle{plain}

    \section{Introduction}
\label{sec:rom:intro}

Partial Differential Equations (PDEs) are very powerful mathematical tools used to model and understand countless phenomena. Although they were originally introduced to describe the laws of nature (e.g., in physics and chemistry), they soon became relevant in many other areas, such as economics or, more recently, as an analytical tool to shed light on the training and behavior of machine learning algorithms. Thanks to their explanatory power, they are often preferred over black-box approaches in many design and decision-making processes.
\new{Using PDEs to solve optimal control problems, inverse problems, or quantifying uncertainty often requires evaluating solutions across a wide range of problem-dependent parameters, rendering the process computationally expensive. Consequently, the primary challenge lies in solving these systems both accurately and efficiently.}

Solving each PDE problem with classical discretization methods is, in general, not a viable option: such solvers are very accurate and reliable, but they are also computationally very expensive when the physical model becomes complex. In this context, the field of Reduced Order Modeling (ROM) studies strategies to build inexpensive mappings that take us from the parameters to good approximations of the solution. To obtain \corr{a mapping which is fast to evaluate,} the main idea is to find a ``compressed yet accurate'' representation of the set of solutions as the parameters vary within a range relevant to the problem. This is more or less challenging depending on the nature of the PDE, and also on the final goal we are pursuing: uncertainty quantification, optimal control, or inverse state or parameter estimation.

This tutorial gives a short introduction to the topic of ROM. We assume that the reader has basic notions of PDEs (theory and numerics, in particular the Galerkin method, see e.g.~\cite{EG2013}), and of functional analysis (familiarity with Hilbert spaces, and the concept of orthogonal projection). Since the field is currently very dynamic, and there are many recent developments, we will focus on what we consider to be the main ideas and the main lines of thought to navigate the vast literature. We will systematically point to further references for the readers wishing to deepen their knowledge in each aspect. Since many statements are very hard to prove, and there are many open questions, we give priority to results that come with theoretical guarantees, but we also present developments which seem to work well in practice despite a lack of completely rigorous justification.

In \cref{sec:fwd}, we start by formulating the problem of finding ``compressed'' representations of parametric PDEs. \corr{We will see that the task can be understood as a specific type of supervised learning problem} where the notion of linear and nonlinear approximation plays an important role in the categorization of existing algorithms. We will see in \cref{sec:fwd:linear} that a lot is known for linear approximations, and in particular that they work very well for elliptic and parabolic PDEs.
\new{However, for other types of PDEs, nonlinear approximation strategies are often needed, especially in cases where solutions present spatial discontinuities or steep gradients. We will see that recent advances leverage mostly two ingredients: i) machine learning tools as a means for devising nonlinear parametric approximations, and ii) transformations inspired by physical insights that allow to fall back to linear model reduction methods.}
In this respect, significant theoretical gaps remain, driving an intense research activity on the topic. We briefly discuss some strategies in \cref{sec:fwd:nonlinear}. We next move on to \cref{sec:inverse} where we explain how to leverage the concepts from \cref{sec:fwd} to solve inverse problems. Such problems arise when we observe a real physical system through partial observations, and we want to build an approximation of the whole system assuming that a certain PDE model is a good description of the underlying physics. In this setting, the nature of the PDE is known, but the parameters explaining at best the observations are not known. We thus combine the knowledge of real data observations with the PDE model to find an optimal approximation of the system. Like before, we divide the discussion into linear and nonlinear approximation.

This tutorial comes with a companion jupyter notebook, which can be found on the following link:
\begin{center}
    \url{https://github.com/agussomacal/MLbook}
\end{center}

\section{Reduced Order Modeling of Parametric PDEs}
\label{sec:fwd}
Parametric PDEs are written in abstract form as
\begin{equation*}
    \cP(u,\theta)=0,
    \label{eq:genpar}
\end{equation*}
where $\cP$ is a partial differential operator, and $\theta=(\theta_1,\dots,\theta_p)$ is a vector of
scalar parameters ranging in some domain $\Theta\subset \bR^p$. We assume well-posedness, that is, for every $\theta\in \Theta$  the problem admits a unique solution $u=u(\theta)$ which lives in a Hilbert space $V$. Typical spaces for $V$ would be $L^2(\Omega)$ or $H^1(\Omega)$ where $\Omega$ is a spatial domain\footnote{Although the domain $\Omega$ often refers to spatial variables it is not necessarily restricted to that meaning as it may also represent other quantities such as time or momentum.} of $\bR^d$. Our notation thus means that $u(\theta):\Omega\to \bR$ is a function of a spatial variable $x\in \Omega$. In what follows, the norm in $V$ and the scalar product are respectively denoted by $\norm{\cdot}_V$ and $\inner{\cdot,\cdot}_V$. The Euclidean norm and scalar product in $\bR^d$ will be denoted as $\abs{\cdot}$ and $\prt{\cdot,\cdot}$.

Our object of interest is the parameter-to-solution mapping
\begin{align*}
    u:\Theta & \to V             \\
    \theta   & \mapsto u(\theta)
    \label{eq:solmap}
\end{align*}
Its image
\begin{equation}
    \cM\coloneqq u(\Theta)=\{u(\theta) \, : \, \theta\in \Theta\}\subset V
    \label{eq:manifold}
\end{equation}
is the set of all PDE solutions when the parameters $\theta$ take values in $\Theta$. Throughout this document, we assume that $\Theta$ is compact in $\bR^p$ and that the map $u$ is continuous and injective (if $u(\theta_1)=u(\theta_2)$, then $\theta_1=\theta_2$). Therefore $\cM$ is a compact set of $V$. In the literature, $\cM$ is often referred to as the \emph{solution manifold}, since it may be thought of as a parameterized $p$-dimensional set immersed in the Hilbert space $V$. Note however that \corr{calling $\cM$ a manifold is not always technically correct since the parameter-to-solution mapping may not satisfy all the axioms of topological manifolds (e.g.~continuity may fail).}

As a guiding example, we consider the PDE operator $\cP$ defined by the Poisson equation \corr{on a bounded domain $\Omega$,}
\begin{align}
    \label{eq:poisson}
    -\nabla\cdot(a\nabla u) & = f, \quad \text{ in }\Omega,         \\
    u                       & =0, \quad \text{ on } \partial\Omega,
\end{align}
where $f:\Omega\to \bR$, and
\begin{align*}
    a = a(\theta) = \sum_{i=1}^p \theta_i \psi_i,\quad \text{with }\theta=(\theta_i)_{i=1}^p,
\end{align*}
where $\psi_i:\Omega\to \bR$ are given functions for $i\in\{1,\dots,p\}$. This notation implies that $a(\theta)$ is a function that depends on the spatial variable $x$ (just in the same manner as $u(\theta)$). The divergence and the gradient in \eqref{eq:poisson} are taken with respect to $x$ (and not $\theta$). \new{We refer to \cref{fig:poisson} for an example solution. The setting explained in the figure is the one which we propose to experiment with in the corresponding jupyter notebook.}

\begin{figure}[ht]
    \centering
    \includegraphics[width=0.65\linewidth]{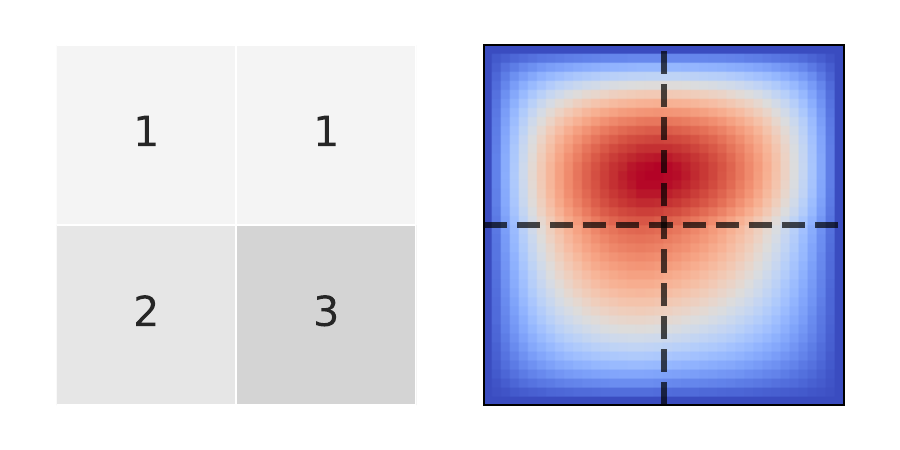}
    \caption{Illustration of a solution of \eqref{eq:poisson} with $\Omega=(0,1)^2$ and $p=4$. In the example, we build $a(\theta)$ based on a partition of $\Omega$ into $4$ smaller squares $\Omega_1=(0, 0.5)\times(0, 0.5]$, $\Omega_2=(0.5, 1.)\times (0, 0.5]$, etc. In the left figure, we see the value of each $\theta_i$ in association with each subdomain $\Omega_i$ so that $a(\theta) = 2\indicative_{\Omega_1}+3\indicative_{\Omega_2}+\indicative_{\Omega_3}+\indicative_{\Omega_4}$ with $\theta=(2,3,1,1)$, and for each $i=1,\dots,4$ we have chosen $\psi_i=\indicative_{\Omega_i}$ which is the indicator function in the corresponding subdomain $\Omega_i$. The corresponding solution $u(\theta)$ is given on the right. In the jupyter notebook, we generate the manifold $\cM$ by making $\theta$ vary in a certain compact set $\Theta \subset \bR^4$, and we propose numerical experiments to understand the model reduction strategies that we present in this tutorial.
    }
    \label{fig:poisson}
\end{figure}

\corr{To proceed, we need to derive a variational formulation of \cref{eq:poisson}. This requires working in the ambient space
    $$
    V = H^1_0(\Omega) \coloneqq \set{v\in H^1(\Omega) \cond v\vert_{\partial \Omega}=0},
    $$
    which is a Hilbert space when equipped with the scalar product, and induced norm
    $$
    \inner{w, v}_{H^1_0(\Omega)} \coloneqq \int_{\Omega} w(x)v(x)\d x + \int_{\Omega} \nabla w(x)\cdot \nabla v(x)\d x, \quad \Vert v \Vert_{H^1_0(\Omega)}\coloneqq \left< v, v\right>_{H^1_0(\Omega)}^{1/2}, \quad \forall w, v \in H_0^1(\Omega).
    $$
}
The variational formulation reads:
\begin{equation}
    \label{eq:vf}
    \text{Find }u\in V=H^1_0(\Omega) \;\st\; a(u,v)=F(v), \quad \forall v\in V,
\end{equation}
where
\begin{equation*}
    \begin{cases}
        a(u, v) & = \sum_{i=1}^p \theta_i \int_\Omega \psi_i \nabla u \cdot \nabla v,
        \quad \forall (u,v)\in V\times V,                                             \\
        F(v)    & =\int_\Omega fv, \quad \forall v\in V.
    \end{cases}
\end{equation*}
Assuming that $f\in L^2(\Omega)$, \corr{$a(\theta)\in L^\infty(\Omega)$,} and that there exists $0<r\leq R$ such that
\begin{equation}
    \label{eq:UEA}
    0<r\leq a(\theta)(x) \leq R, \quad \forall x\in \Omega,
\end{equation}
then the Lax-Milgram theorem guarantees that, for every $\theta\in \Theta$, problem \eqref{eq:vf} has a unique solution $u(\theta) \in V = H^1_0(\Omega)$. For clarity of notation, when we specifically refer to the solution set of our Poisson example, we will use the notation $\cMpoisson$.

A very common way of finding a numerical approximation of $u(\theta)$ is the Galerkin method: we replace the infinite dimensional space $V$ in \eqref{eq:vf} by a finite dimensional subspace $V_h\subset V$, and consider the variational problem
\begin{equation}
    \label{eq:vf-galerkin}
    \text{Find }u_h\in V_h \st \quad a(u_h,v_h)=F(v_h), \quad \forall v_h\in V_h.
\end{equation}
When $V_h$ is chosen as a subspace of piecewise polynomials \new{defined on an underlying spatial grid}, this method is called the Finite Element Method (FEM). The parameter $h$ refers to the size of the mesh that is used, and we denote $N=\dim V_h$ the dimension of $V_h$.

\subsection{Do we really need ROM?}
\label{sec:need-for-rom}
We are now going to recall a few key results from finite element analysis \corr{whose aim is to motivate why the dimension $N$ of $V_h$ is expected to be typically large in the following three situations:
    \begin{enumerate}
        \item when one requires a high degree of accuracy in the solution,
        \item when the space dimension is large ($d\gg 1$),
        \item when there are local spatial discontinuities.
    \end{enumerate}
    The complexity of solving the Galerkin formulation \eqref{eq:vf-galerkin} grows at least as $\ord(N)$ (in general, it will be even worse, of the order $\ord(N^\alpha)$ for some $\alpha>1$). So if $N\gg 1$, and we need to solve the PDE for many parameter instances, the expected complexity becomes very large, and this motivates the need to build another type of algorithm that maps parameters to solutions in a more affordable manner, which is precisely the main goal in ROM .
}

To start discussing why $N$ is expected to be large in the above three situations, let us first recall Céa's lemma. It connects the approximation error given by the Galerkin method with the best approximation in $V_h$ through the following bound.

\begin{lemma}[Céa's Lemma]
    \new{For the Poisson problem from \eqref{eq:vf},} the accuracy of the Galerkin approximation \eqref{eq:vf-galerkin} is bounded by
    \begin{equation*}
        \norm{u-u_h}_{H^1_0(\Omega)}\leq \corr{\prt*{\frac{2R}{r\min(1, C_\Omega^{-2})}}^{1/2}} \min_{v_h\in V_h} \norm{u - v_h}_{H^1_0(\Omega)},
    \end{equation*}
    \corr{where $C_\Omega>0$ is the Poincaré constant associated to $\Omega$.}
\end{lemma}

By a nontrivial development, we can leverage this bound to connect the size of the mesh with the accuracy of the solution if we have some knowledge on the regularity of $u$.
There are plenty of results of this flavor (see for example \cite{babuskaRatesConvergenceFinite1982,suli2012lecture}), and we record one of them in the next theorem (see \cref{eq:conv-fem-rate}).

\begin{theorem}[Convergence of FEM]
    \label{thm:conv-fem}
    Let $u\in H^1_0(\Omega)$ and $u_h\in V_{h}$ be the solutions of \eqref{eq:vf-galerkin}. Then, the finite element method involving piece-wise polynomials of degree $m\geq 1$ converges, that is,
    \begin{equation}
        \label{eq:conv-fem}
        \lim_{h\to 0} \Vert u - u_h\Vert_{H^1_0(\Omega)}=0.
    \end{equation}
    Moreover, if $u\in H^{m+1}(\Omega)$, then there exists a constant $C>0$ independent of $h$ such that
    \begin{equation}
        \label{eq:conv-fem-rate}
        \Vert u - u_h\Vert_{H^1_0(\Omega)} \leq C h^m \vert u \vert_{H^{m+1}(\Omega)},
    \end{equation}
    where
    $$
    \vert u \vert_{H^{m+1}(\Omega)}
    \coloneqq
    \prt*{\sum_{|k|=m+1} \norm{\partial^k v }^2_{L^2(\Omega)} }^{1/2}
    $$
    is the $H^{m+1}(\Omega)$ semi-norm of $u$.
\end{theorem}

Equation \eqref{eq:conv-fem} from Theorem \ref{thm:conv-fem} guarantees convergence for any degree of the finite element. Equation \eqref{eq:conv-fem-rate} gives the rate of convergence, and reveals that the rate depends on the size $h$ of the mesh, and on the spatial regularity of the solution. We gain one order of convergence per polynomial degree. If the solution is not regular enough ($u\not\in H^{m+1}(\Omega)$), then the results \new{from \cref{thm:conv-fem}} show that there is no theoretical advantage to use finite elements of degree $m$ instead of elements of lower degree. \corr{However, a finer analysis may lead to a more nuanced landscape depending on the situation since the presence of local spatial regularity may justify the use of higher degree in subdomains where the solution is regular.}

Using finite elements of high degree comes at the price of having to deal with more unknowns \new{and denser systems of linear equations to invert}. Let us discuss the concrete implications when we work with a \textbf{quasi-uniform triangulation} of a domain $\Omega\subset \bR^d$. In this case, the dimension $N=\dim(V_h)$ is of the same order of the number of cells, that is
$$
N \sim h^{-d}.
$$
We can thus reformulate the estimate \eqref{eq:conv-fem-rate} as a compromise between accuracy and the number $N$ of degrees of freedom:
\begin{equation}
    \label{eq:fem-conv-dof}
    \Vert u - u_h\Vert_{H^1_0(\Omega)} \leq C \vert u \vert_{H^{m+1}(\Omega)} N^{-\frac m d}.
\end{equation}
This means, concretely, that to reach a given accuracy $\eps$, the number $N(\eps)$ of degrees of freedom needed is of order
$$
N(\eps)\sim \left(\frac{C \vert u \vert_{H^m(\Omega)}}{\eps}\right)^{d/m}.
$$
It grows exponentially with the dimension, and this fact is called the \emph{curse of dimensionality}. We see that smoothness has the opposite effect: the rate improves with regularity which is sometimes called the \emph{blessing of smoothness}.

To understand the drastic implications of this curse of dimensionality, it is good to study an example. Suppose we fix the smoothness $m$, and suppose that we currently have an approximation $u_h$ with $N(\eps)= \eps^{-d/m}$ degrees of freedom giving an error of $\eps$. Suppose now that we want to improve the quality of our approximation, i.e.\ reduce the error $\eps$,  by a factor of $2$. This is a fairly modest improvement, but the expression $\eps= N^{-m/d}$  tells us that we need $N'=N(\eps/2)$ degrees of freedom with
\[
    (N')^{-m/d} = \frac12 \eps = \frac12 N^{-m/d}
    \qquad\text{or equivalently} \qquad
    \frac {N'}N = 2^{d/m}.
\]
If we are working in dimension $2$ and $m=1$ then we need to take $4$ times as many points. If $\Omega$ has $10$ dimensions, then we need to take $ 2^{10} = 1024$  as many points, which typically already is prohibitive. When $d=100$, $2^d \approx 10^{30}$, and when $d=1000$, $2^d\approx 10^{301}$. This is more than the cube of the amount of estimated atoms in the universe ($10^{80}$). The number $d=1000$ may seem artificial given the low dimensional examples that we consider in our programming tutorial. However, PDEs in very high dimension arise in applications such as molecular dynamics where a simulation of $k$ particles in spatial dimension $D$ yields PDEs with effective dimension $d=k^D$.

The curse of dimensionality is mitigated when we have smoothness. However, the higher the dimension, the more smoothness we need to require, and at some point this is no longer a realistic assumption. In addition, even for small dimensions, we may not have a lot of regularity. As a consequence, $\vert u \vert_{H^{m+1}(\Omega)}$ in \eqref{eq:fem-conv-dof} becomes very large or even infinite even for small $m$. This happens when the solution has local singularities, which typically occurs when the right-hand side $f$ has singularities, when the PDE coefficients have discontinuities or when the boundary conditions are of Dirichlet type in one part, and of Neumann type in another part. In such situations, non-uniform meshes refined near singularities become essential. These discretizations are typically constructed using a posteriori error estimators.

\corr{More broadly, one may consider approximation strategies beyond finite elements. However, regardless of the particular method, the same challenge remains: the number $N$ of degrees of freedom can become extremely large in high-dimensional settings, in problems with limited regularity, or when very high accuracy is required. Consequently, when solutions must be computed repeatedly for many parameter values, it is natural to seek fast surrogate models that avoid reconstructing such costly discretizations for every new query. This is precisely the main motivation behind reduced-order modeling (ROM).}

\corr{Although this discussion could be developed further, we stop here to keep the focus on ROM. We now return to the main thread, hoping that the above considerations have made clear why the number of degrees of freedom $N$ should be thought as a large number which leads to challenges in a parametric setting.}

\subsection{Linear Reduced Order Modeling}
\label{sec:fwd:linear}
In general, for every $\theta\in \Theta$, one cannot exactly compute $u(\theta)$ unless the solution to the PDE has an analytic form. As explained just above, to approximate $u(\theta)$ we may rely on classical discretizations such as finite element methods which would compute an approximate solution in a finite element space $V_h\subset V$ of dimension $N\gg 1$. The goal of ROM is to find an approximation space $V_n$ of dimension $n$ that approximates $u(\theta)$ at a comparable accuracy than the one of $V_h$ but with a big dimension reduction $n\ll N$. The space $V_n$ could either be an $n$-dimensional linear subspace of $V$, or more generally a nonlinear set parametrized by $n$ degrees of freedom. As we will see in the following, sometimes $V_n$ is actually built as a subspace of a finite element space $V_h$.

Methods of approximation are essentially divided into two classes: linear and nonlinear. In linear approximation, $V_n$ is a linear subspace of $V$, and we want to choose it so that it approximates well all the elements $u\in \cM$. Once such a space is found, we can then build a numerical strategy
\begin{equation*}
    \label{eq:widehat-u}
    \widehat u_n^{\lin}:\Theta \to V_n
\end{equation*}
such that $\widehat u_n^{\lin}(\theta)$ approximates $u(\theta)$, and such that the numerical complexity to evaluate the mapping $\widehat u_n^{\lin}$ is of order $\ord(n^\alpha)$ instead of $\ord(N^\alpha)$. For example, to solve problem \eqref{eq:poisson} one could use the Galerkin method corresponding to $V_n$.

Let us examine how much accuracy we can expect when working with linear subspaces $V_n$. For each $\theta\in \Theta$, the best approximation of $u(\theta)$ in $V_n$ is
\begin{equation*}
    e(u(\theta), V_n) \coloneqq \inf_{v\in V_n} \norm{u(\theta)-v}_V,
\end{equation*}
which, for $V_n$ a finite dimensional subspace, it can be written as
\begin{equation*}
    e(u(\theta), V_n) = \norm{u(\theta)-P_{V_n}u(\theta)}_V.
\end{equation*}
where $P_{V_n}:V\to V_n$ is the orthogonal projection operator onto $V_n$. The quality of $V_n$ for our parametric problem is either given by
\begin{equation}
    \label{eq:wc-lin-rom}
    \cE^{\worstcase}(\cM; V_n) \coloneqq \max_{\theta \in \Theta} e(u(\theta), V_n)
\end{equation}
or by
\begin{equation}
    \label{eq:ac-lin-rom}
    \cE^{\averagecase}_\rho(\cM; V_n) \coloneqq \left(\int_{\theta \in \Theta} e^2(u(\theta), V_n) \rho(\dd \theta) \right)^{1/2}.
\end{equation}
We refer to these quantities as the error on the class $\cM$. The first quantity \eqref{eq:wc-lin-rom} measures accuracy in the worst case sense. The second one is based on the average accuracy over all elements when we place ourselves in a probabilitic setting in which we assume that the parameters follow a certain distribution $\rho\in \cP(\Theta)$. We have that
$$
\cE^{\averagecase}_\rho(\cM; V_n) \leq \cE^{\worstcase}(\cM; V_n), \quad \forall \rho \in \cP(\Theta).
$$

The best choice of a linear space $V_n$ is the one that minimizes the error \cref{eq:wc-lin-rom} or \cref{eq:ac-lin-rom}. This smallest error is called the Kolmogorov $n$-width of $\cM$, and it is defined as (see \cite[Chapter 3]{BCOW2017})
\begin{equation*}
    \d_n^{\star}(\cM) \coloneqq \inf_{\substack{W \text{linear}\\ \dim(W)\leq n}} \cE^{\star}(\cM; W) , \quad \text{ with } \star = \{\averagecase, \worstcase\}.
\end{equation*}

The Kolmogorov $n$-width gives the optimal performance that we can achieve when working with a linear approximation space of dimension $n$ to approximate the elements of $\cM$. Determining the value of $d^\star_n(\cM)$ for a given solution set $\cM$, and finding an optimal or near-optimal subspace is a very difficult problem. A salient example where a rate for $d^\worstcase_n(\cM)$ is known is our example on the parametric Poisson problem.

\begin{theorem}[see, e.g., \cite{TWZ2017}]
    \label{thm:kolmo-width-poisson}
    For the parametrized Poisson problem \eqref{eq:vf}, if the uniform ellipticity assumption \eqref{eq:UEA} holds, then
    $$
    d_n^\worstcase\prt*{\cMpoisson} \leq C e^{-c n^{1/d}},
    $$
    where $C,\, c>0$.
\end{theorem}

\begin{remark}
    \corr{Beyond our elliptic problem, there are a few other parametric PDE manifolds for which the decay of the Kolomogorov width has been estimated. We refer to \cite{AGU2025} for the study of manifolds  $\cM^g=\set{g(\cdot-\theta)\cond \theta \in [0,1]}$ generated by translations of a ``template function'' $g\in H^r(\Omega)$. Such manifolds arise in parametric transport problems. It is proven in \cite{AGU2025} that $d_n(\cM^g)\lesssim n^{-r}$ (the bound is slightly better than that but it would become technical to present it here, see Theorem 5.3 of \cite{AGU2025}). For wave-type equations, some results have been derived in \cite{GK2019}.}
\end{remark}

Since finding the optimal space is out of reach, we next present existing methods to find good approximation spaces $V_n$ in practice.

\subsubsection{The average case}
When working in the average sense, the space $V_n^{\opt}$ which is the minimizer of $d^{\averagecase}_n(\cM)$ is completely characterized by a singular value decomposition. To explain it, let us recall a few concepts from the spectral theory of compact operators.

\begin{theorem}[Spectral Theorem]
    \label{thm:spectral-thm}
    Let $H$ be a Hilbert space with norm $\norm{\cdot}_H$ and inner product $\inner{\cdot,\cdot}_H$, and let $A: H \to H$ be a self-adjoint, compact, linear operator. Then there exists a system of orthonormal vectors $\{h_i\}_{i\in I}\subset H$ and corresponding real eigenvalues $\{\lambda_i\}_{i\in I}\subset \bR$ with $\abs{\lambda_1}\geq \abs{\lambda_2}\geq \dots$ such that
    $$
    A(h_i) = \lambda_i h_i, \quad \forall i \in I,
    $$
    and $A$ admits an \emph{eigendecomposition} of the form
    $$
    A(h) = \sum_{i\in I} \lambda_i \inner{h_i, h}_H h_i, \quad \forall h\in H.
    $$
    The set $I$ is either finite or countably finite so we work with $I\subseteq \bN$. When $I=\bN$, then $\lambda_n\to 0$ as $n\to \infty$.
\end{theorem}

Note that when $I=\bN$, the system $\{h_i\}_{i\in I}$ forms an orthonormal basis of $H$, and when $I$ is finite, the system can be completed (e.g.~by Gram-Schmidt orthonormalization) to form an orthonormal basis of the full space $H$. As a consequence of \cref{thm:spectral-thm}, we can deduce the following Lemma. We give its proof for the sake of completeness.

\begin{lemma}[Singular Value Decomposition]
    \label{lem:svd}
    Let $H$ and $F$ be two Hilbert spaces (with our usual notation for norms and inner products), and let $B: H \to F$ be a compact, linear operator. Then:
    \begin{enumerate}
        \item $A=B^*B:H\to H$ is self-adjoint, compact and linear with countable eigenvalues $\{\lambda_i\}_{i\in I}$ and corresponding orthonormalized eigenvectors $\{h_i\}_{i\in I}\subset H$.
        \item The set $\{f_i\}_{i\in I}$ with $f_i\coloneqq \lambda_i^{-1/2} B(h_i)$ forms an orhonormal system of $F$.
        \item $B$ has a \emph{singular value decomposition} of the form
        $$
        B(h) = \sum_{i\in I} \lambda_i^{1/2} \inner{h, h_i}_H f_i.
        $$
    \end{enumerate}
\end{lemma}

\begin{proof}
    \begin{enumerate}
        \item Since $B$ is compact and linear, its adjoint is also compact and their composition $A=B^*B$ is also compact. By \cref{thm:spectral-thm}, $A$ has an eigenvalue decomposition.
        \item For any $(i,j)\in I^2$,
        $$
        \inner{f_i, f_j}_F = (\lambda_i \lambda_j)^{-1/2} \inner{B(h_i), B(h_j)}_F = (\lambda_i \lambda_j)^{-1/2} \inner{h_i, B^*B(h_j)}_H = \delta_{i,j}
        $$
        \item Since $\Im(B)=\vspan\{B(h_i)\}_{i\in I}=\vspan\{f_i\}_{i\in I}$, and that the $f_i$ are orthonormal, we have for all $h\in H$,
        \begin{align*}
            B(h) & = \sum_{i\in I} \inner{B(h), f_i}_F f_i                                                                                                      \\
            & = \sum_{i, j \in I} \inner{h, h_j}_H \inner{B(h_j), f_i}_F f_i \quad \text{since } h = \sum_{j\in I} \inner{h,h_j}_H h_j                     \\
            & = \sum_{i, j \in I} \inner{h, h_j}_H \lambda_i^{1/2} \delta_{i,j} f_i \quad \text{since } \inner{B(h_j), f_i}_F=\lambda_i^{1/2} \delta_{i,j} \\
            & = \sum_{i\in I} \lambda_i^{1/2} \inner{h,h_i}_H f_i.
        \end{align*}
    \end{enumerate}
\end{proof}

We now leverage these results to characterize a space $V_n^{\opt}$ that is a minimizer of $d^{\averagecase}_n(\cM)$. The development requires to assume that our mapping $u:\Theta\to V$ is such that $u\in L^2((\Theta, \rho), V)$. In other words, we need that
\begin{equation}
    \label{eq:u-param-integrability}
    \norm{u}^2_{L^2((\Theta, \rho), V)} =
    \int_{\theta\in\Theta} \norm{u(\theta)}^2_V \rho(\dd\theta)<+\infty.
\end{equation}
To formulate the result, we need to introduce the operator
\begin{align}
    R : L^2(\Theta, \rho) & \to V                                                                                                                                                                           \\
    \varphi                & \mapsto R(\varphi)(x) \coloneqq \int_{\Theta} u(\theta)(x) \varphi(\theta) \rho(\dd\theta) = \inner{u(\cdot)(x), \varphi(\cdot)}_{L^2(\Theta,\rho)}, \quad \forall x\in \Omega.
\end{align}

\begin{theorem}
    \label{thm:svd-ppdes}
    Suppose that the mapping $u:\Theta\to V$ is such that $u\in L^2((\Theta, \rho), V)$. Then:
    \begin{enumerate}
        \item $R$ is a bounded linear operator with adjoint
        \begin{align}
            R^*: V & \to L^2(\Theta, \rho)                                                                                    \\
            v      & \mapsto R^*(v)(\theta) = \inner{u(\theta), v}_{V}, \quad \forall \theta \in \Theta. \label{eq:R-adjoint}
        \end{align}
        \item The operator
        \begin{align}
            C\coloneqq R^*R : L^2(\Theta, \rho) & \to L^2(\Theta, \rho)                                                                                                                   \\
            \varphi                             & \to C(\varphi)(\theta)\coloneqq \int_\Theta  k(\theta,\theta') \varphi(\theta')\rho(\dd\theta'),\quad \forall \theta\; \rho\text{-a.e.}
            \label{eq:HS}
        \end{align}
        with
        $$
        k(\theta,\theta')=\inner{u(\theta), u(\theta'}_V,
        $$
        is linear, compact, and self-adjoint, with countable eigenvalues $\{\lambda_i\}_{i\in I}$, and corresponding orthonormalized eigenvectors $\{h_i\}_{i\in I}\subset L^2(\Theta, \rho)$. If necessary, the set of eigenvectors can be completed to form an orthonormal basis $\{h_i\}_{i\in \bN}$ of $L^2(\Theta, \rho)$.
        \item $R$ has a singular value decomposition $$R(\varphi) = \sum_{i\in I} \lambda_i^{1/2} \inner{\varphi, h_i}_{L^2(\Theta, \rho)} f_i,$$ where $f_i\coloneqq \lambda_i^{-1/2} R(h_i)$ and $\{f_i\}_{i\in I}$ forms an orthonormal system in $V$ which can be completed to form an orthonormal basis of $V$.
        \item The space $V_n^{\opt}\coloneqq \vspan\{f_1,\dots, f_n\}$ is optimal for $d_n^{\averagecase}(\cM)$, and
        $$
        d_n^{\averagecase}(\cM) = \left(\sum_{i>n} \lambda_i^{1/2}\right)^{1/2}.
        $$
    \end{enumerate}
\end{theorem}

\begin{proof}
    \begin{enumerate}
        \item $R$ and $R^*$ are linear. They are also bounded because our assumption $u\in L^2((\Theta, \rho), V)$ implies that \eqref{eq:u-param-integrability} holds. In addition, it holds that for all $(\varphi, v)\in L^2(\Theta,\rho)\times V$,
        \begin{equation*}
            \inner{R(\varphi), v}_V
            = \inner*{\int_{\Theta} u(\theta)(\cdot) \varphi(\theta) \rho(\dd\theta) , v(\cdot)}_V
            = \int_{\Theta} \inner{ u(\theta), v}_V \varphi(\theta) \rho(\dd\theta)
            = \inner*{\varphi, \inner*{ u(\cdot), v}_V}_{L^2(\Theta,\rho)}.
        \end{equation*}
        This proves that $R^*(v)=\inner{ u(\cdot), v}_V$ is the adjoint of $R$.
        \item By the Cauchy-Schwarz inequality and assumption \eqref{eq:u-param-integrability},
        \begin{equation*}
            \norm{k}^2_{L^2(\Theta\times\Theta, \rho\otimes \rho)} = \int_{\Theta^2} k^2(\theta,\theta')\rho(\dd\theta)\rho(\dd\theta') \leq \norm{u}^4_{L^2((\Theta, \rho), V)} <+\infty.
        \end{equation*}
        Therefore, $C$ is a Hilbert-Schmidt integral operator, so it is compact, see \cite[Theorem VI.22]{RS1980}. Since $k$ is symmetric, $C$ is also self-adjoint. By \cref{thm:spectral-thm}, $C$ admits an eigenvalue decomposition.
        \item We obtain the statement by following the same lines as for point 3 in \cref{lem:svd} (use $H=L^2((\Theta, \rho))$ and $F=V$).
        \item We now prove that $V_n^\opt = \vspan\{ f_1,\dots, f_n \}_{i=1}^n$ minimizes $d_n^{\averagecase}(\cM)$. By expanding $u(\theta)$ in the orthonormal basis $\{f_i\}_{i\in \bN}$, we have that
        $$
        \cE^{\averagecase}_\rho(\cM; V_n^\opt)^2
        =
        \int_\Theta \Vert u(\theta) - P_{V_n^\opt} u(\theta) \Vert_V^2 \rho(\dd\theta) = \sum_{i>n} \int_\theta | \inner{u(\theta), f_i}_V |^2 \rho(\dd\theta).
        $$
        From \cref{eq:R-adjoint}, $\inner{u(\theta), f_i}_V = R^*(f_i) = \lambda_i^{1/2} h_i$, thus we find that
        $$
        \cE^{\averagecase}_\rho(\cM; V_n^\opt)^2
        = \sum_{i>n} \Vert R^*(v_i) \Vert^2_{L^2(\Theta, \rho)} = \sum_{i>n}  \lambda_i .
        $$
        To see that this is the optimal space for $d_n^{\averagecase}(\cM)$, we proceed as follows. Suppose first that $V_n = \vspan\{ f_{i_1}, \dots, f_{i_n}\}$ and denote $I_n = \{i_1,\dots, i_n\}$. The same computation that we have just done for $V_n^{\opt}$ leads to
        $$
        \cE^{\averagecase}_\rho(\cM; V_n)^2
        = \sum_{i \not\in I_n}  \lambda_i \geq \cE^{\averagecase}_\rho(\cM; V_n^\opt)^2,
        $$
        and we see that this error is minimized when we choose $I_n = \{1, \dots, n\}$ (which corresponds to $V_n^\opt$). More in general, for any $n$-dimensional space $V_n=\vspan\{v_1,\dots, v_n\}$, by using that the basis elements can be expressed as linear combinations of the $\{f_i\}_{i\in I}$, we can prove that
        $$
        \cE^{\averagecase}_\rho(\cM; V_n) \geq
        \cE^{\averagecase}_\rho(\cM; V_n^\opt)
        $$
        for any $n$-dimensional space $V_n$. This shows that
        $$
        V_n^{\opt} \in \argmin_{\substack{\dim(W) \leq n \\ W \subseteq V}} \cE^{\averagecase}_\rho(\cM; W),
        $$
        and concludes the proof.
    \end{enumerate}
\end{proof}

In practice, we find a numerical approximation of $V_n^\opt$ by building a sampled version of the operator $C$ (see equation \eqref{eq:HS}). For this, we draw $K$ independent random samples $\theta_i$ from the probability distribution $\rho$, and we define the empirical measure $\widehat \rho_K \coloneqq \frac 1 K \sum_{i=1}^K \delta_{\theta_i}$. We next define the associated operator
\begin{align*}
    \widehat C_K : L^2(\Theta, \widehat \rho_K) & \to L^2(\Theta, \widehat \rho_K)                                                                                                                             \\
    \varphi                                     & \to \widehat C_K(\varphi)(\theta)\coloneqq \frac 1 K \sum_{j=1}^K k(\theta,\theta_j) \varphi(\theta_j), \quad \forall \theta \; \widehat \rho_K\text{-a.e.}
\end{align*}
By direct application of point 2 of \cref{thm:svd-ppdes}, $\widehat C_K$ has an eigenvalue decomposition with countable eigenvalues $\{\widehat \lambda_i^{(K)}\}_{i\in I}$ and orthonormalized eigenvectors $\{\widehat h_i^{(K)}\}_{i\in I}\subset L^2(\Theta, \widehat \rho_K)$, and they satisfy $\widehat C_K(\widehat h_i^{(K)})(\theta)= \widehat \lambda_i^{(K)}\widehat h_i^{(K)}(\theta)$. Since this relation holds for all $\theta$ $\widehat\rho_K$-almost everywhere, it means that it has to hold for all samples $\theta_i$ for $i=1,\dots, K$. Consequently, each eigenpair $(\widehat \lambda, \widehat h)$ satisfies
\begin{equation}
    \label{eq:CK-eigen}
    \widehat C_K(\widehat h)(\theta_i)= \widehat \lambda\widehat h(\theta_i), \quad \forall i=1,\dots, K
    \quad \Leftrightarrow \quad
    \widehat{\tC}_K \widehat{\textbf{h}} = \widehat \lambda \widehat{\textbf{h}}
\end{equation}
where
$$
\widehat{\tC}_K \coloneqq \prt*{ k(\theta_i,\theta_j)}_{1\leq i, j \leq K}, \quad
\widehat{\textbf{h}} \coloneqq \prt{\widehat h(\theta_i)}_{i=1}^K.
$$
Therefore, the eigenvalues and eigenfunctions of the operator $\widehat C_K$ are given by the eigendecomposition of the matrix $\widehat{\tC}_K$ which is symmetric, positive definite.

From this development, we can derive a practical strategy to build an approximation $\widehat f_i^{(K)}$ of the \corr{orthonormal basis} $f_i$ that span the optimal space $V_n^\opt$:
\begin{enumerate}
    \item Compute $u(\theta_i)$ for the random samples $\theta_i \sim \rho$, and assemble the matrix $\widehat{\tC}_K\in \bR^{K\times K}$. In practice, this is done with a classical discretization solver, so we actually work with $u_h(\theta_i)\in V_h$.
    \item Compute the eigenvalue decomposition of $\widehat{\tC}_K$ which we denote $\widehat{\tC}_K = \widehat{\textbf{U}} \widehat{\bm{S}} \widehat{\textbf{U}}^\top$ where \corr{$\widehat{\textbf{U}}=(\widehat{U}_{i,j})_{1\leq i,j\leq K})\in \bR^{K\times K}$} is the orthogonal matrix carrying the eigenvectors on each column, and $\widehat{\bm{S}}=\diag(s_i)\in \bR^{K\times K}$ is the diagonal matrix containing the nonnegative eigenvalues.
    \item From \cref{eq:CK-eigen}, it follows that $s_i = \widehat \lambda_i^{(K)}$ for $i=1,\dots, K$, and that \begin{equation}
                                                                                                                     \widehat h_i^{(K)}(\theta)=
                                                                                                                     \begin{cases}
                                                                                                                         \widehat U_{j,i}, & \quad \text{ if }\theta=\theta_j \\
                                                                                                                         0,                & \quad \text{ otherwise}.
                                                                                                                     \end{cases}
    \end{equation}
    \item We can finally define the approximation of $f_i$ as
    \begin{equation}
        \widehat f_i^{(K)}\coloneqq s_i^{-1/2} \widehat R_K\prt{\widehat h_i^{(K)}} = \frac{s_i^{-1/2}}{K}\sum_{j=1}^K u_h(\theta_j) \widehat h_i(\theta_j),
    \end{equation}
    and assuming that $K\geq n$, we can consider the space $\widehat V_n \coloneqq \vspan\{\widehat f_i^{(K)}\}_{i=1}^n$ as a surrogate of $V_n^\opt$.
    \item In practice, we usually work with evaluations of $\widehat f_i^{(K)}$ at points $\{x_1,\dots, x_N\} \subset \Omega$. This can be done by matrix multiplication as
    $$
    \widehat{\bm{F}} = \widehat{\bm{S}}^{-1/2} \widehat{\bm{U}}^\top \widehat{\bm{T}},
    $$
    where
    $$
    \widehat{\bm{F}} = \prt*{ \widehat f_i^{(K)}(x_j) }_{\substack{1\leq i \leq n \\ 1\leq j \leq N}},
    \quad\text{and}\quad
    \widehat{\bm{T}} = \prt*{ u_h(\theta_i)(x_j) }_{\substack{1\leq i \leq K \\ 1\leq j \leq N}}.
    $$
\end{enumerate}

\begin{remark}[Historical Note]
  The singular value decomposition (SVD) has its roots in early functional analysis and matrix approximation theory, beginning with Schmidt’s work on compact operators (\cite{schmidt1907zur}), and the optimal low-rank approximation theorem of Eckart and Young (\cite{eckart1936approximation}). It later became closely connected to the Karhunen–Loève expansion (\cite{karhunen1947uber,loeve1948fonctions}), principal component analysis (PCA, see \cite{pearson1901lines, hotelling1933analysis}), and proper orthogonal decomposition (POD), eventually emerging as a central computational tool in reduced-order modeling \cite{sirovich1987turbulence1,berkooz1993pod,holmes1996turbulence}.
\end{remark}

\subsubsection{The worst case}
\label{sec:plain-greedy}
For the worst case sense there is no explicit characterization of an optimal space $V_n^{\opt}$ that is a minimizer of $d^{\worstcase}_n(\cM)$. However, one can build subspaces $V_n$ for which $\cE^{\worstcase}(\cM, V_n)$ decays at a comparable rate as the Kolmogorov $n$-width $d_n^\worstcase(\cM)$ as $n\to \infty$.
\new{
    The construction of such a space is done with the following iterative algorithm (see \cite{BCDDPW2011, DPW2013}).

    \textbf{Greedy algorithm}: For $n=1$, we find
}
\begin{equation}
    \label{eq:vanilla-greedy-n1}
    u_1 \in \argmax_{u\in \cM} \Vert u \Vert_V
\end{equation}
and define
$$
\rU_1 \coloneqq \{u_1\}, \quad V_1\coloneqq \vspan\{u_1\}.
$$
For $n>1$, given $\rU_{n-1}$ and $V_{n-1}$, we find
\begin{equation}
    \label{eq:vanilla-greedy-n}
    u_n \in \argmax_{u\in \cM} \min_{v\in V_{n-1}} \Vert u -v\Vert_V
\end{equation}
and set
$$
\rU_n = \rU_{n-1} \cup \{u_n\}, \quad V_n = \vspan\{ \rU_n \}.
$$

\new{
    The maximization problems \eqref{eq:vanilla-greedy-n1}-\eqref{eq:vanilla-greedy-n} defined over the full manifold $\cM$ are not computable in general.
    For this reason, in practice the maximization is performed over the finite set $\cM^K\coloneq \{u(\theta_1), \dots, u(\theta_K)\}$. The weak greedy version of \eqref{eq:vanilla-greedy-n1}-\eqref{eq:vanilla-greedy-n} allows us to get guaratees for a family of greedy methods including the former algorithm defined on $\cM$ and the latter empirical version using $\cM^K$ instead of $\cM$.
}

\corr{
    \textbf{Weak greedy algorithm}: for a parameter $\gamma\in (0, 1]$, we search for
}
\begin{equation}
    u_1 \in \{ u \in \cM\cond \Vert u \Vert_V \geq \gamma \sup_{v\in \cM} \Vert v \Vert_V  \}
\end{equation}
and for $n>1$,
\begin{equation*}
    u_n \in \{ u \in \cM\cond \inf_{v\in V_{n-1}}\Vert u-v \Vert_V \geq \gamma \sup_{z\in \cM} \inf_{v\in V_{n-1}} \Vert z-v \Vert_V  \}.
\end{equation*}
The parameter $\gamma$ models how close an optimization method can approach the real optimum. When $\gamma=1$, we recover the original greedy algorithm.

The following result shows that the weak-greedy algorithm generates a sequence of spaces $\{V_n\}_n$ for which $\cE^{\worstcase}(\cM; V_n)$ decays at a comparable rate as the Kolmogorov $n$-width $\d_n^\worstcase(\cM)$.

\begin{theorem}[see Corollary 3.3 of \cite{DPW2013}]
    \label{thm:conv-greedy}
    Consider the weak greedy algorithm with parameter $\gamma\in (0,1]$. Then:
    \begin{itemize}
        \item For any $C_0>0$ and $s>0$, we have
        \begin{equation}
            d_n^\worstcase(\cM) \leq C_0 n^{-s},\; n\geq 1 \quad \Longrightarrow \cE^{\worstcase}(\cM, V_n) \leq C_1 n^{-s}, \; n\geq 1,
        \end{equation}
        with $C_1 \coloneqq 2^{5s+1}\gamma^{-2}C_0$.
        \item For any $c_0, C_0>0$ and $s>0$, we have
        \begin{equation}
            d_n^\worstcase(\cM) \leq C_0 e^{-c_0 n^{s}},\; n\geq 1 \quad \Longrightarrow \cE^{\worstcase}(\cM, V_n) \leq \sqrt{2C_0} \gamma^{-1} e^{-c_1 n^{s}}, \; n\geq 1,
        \end{equation}
        where $c_1 = 2^{-1-2s}c_0$.
    \end{itemize}
\end{theorem}

Chaining \cref{thm:kolmo-width-poisson} with \cref{thm:conv-greedy}, we directly obtain that the weak greedy has a sub-exponential decay rate for the Poisson problem. We record this fact in the following corollary.

\begin{corollary}
    \label{cor:poisson-weak-greedy}
    The weak greedy algorithm with parameter $\gamma\in(0,1]$ applied to $\cMpoisson$ gives
    $$
    \cE^{\worstcase}(\cMpoisson, V_n) \leq \sqrt{2 C} \gamma^{-1} e^{- 2^{-1-\frac{2}{d}}c n^{1/d}}.
    $$
    where $C,\,c$ are the constants involved in  \cref{thm:kolmo-width-poisson}.
\end{corollary}

\begin{figure}[ht]
    \centering
    \includegraphics[width=0.5\linewidth]{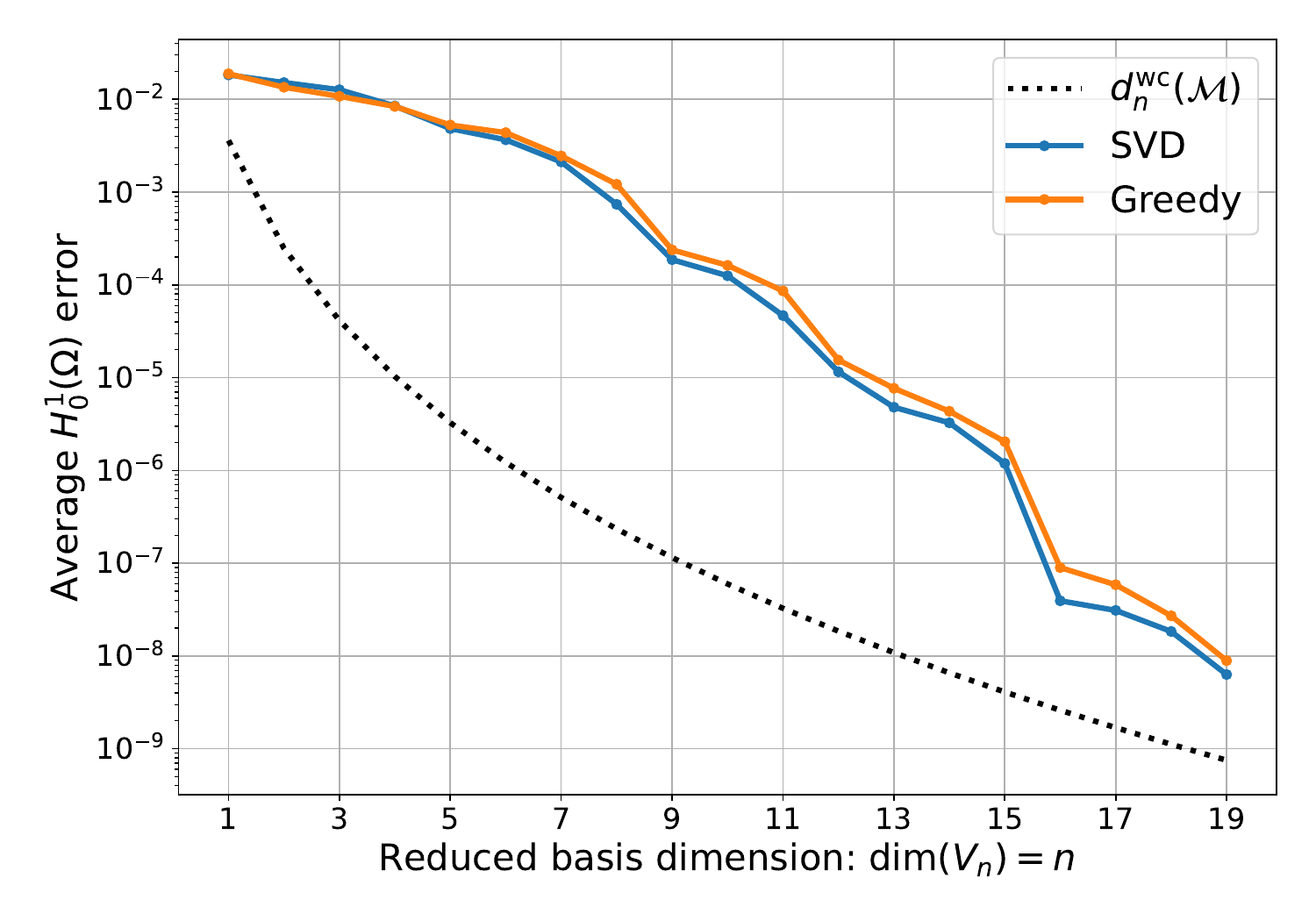}
    \caption{Estimates of the $\cE^{\worstcase}(\cM; V_n)$ and $\cE^{\averagecase}_\rho(\cM; V_n)$ for the Poisson example \eqref{eq:poisson} with $p=4$, and when $\psi_i=\indicative_{\Omega_i}$ are constant functions in their corresponding subdomains $\Omega_i$ (as we have seen in \cref{fig:poisson}). The estimates are obtained by generating a discrete training set of $K=200$ samples of $u(\theta)$, which was obtained by uniformly sampling $\theta\sim [1, 10]^4$. We then run the greedy algorithm, and compute the SVD on this discrete set to estimate $\cE^{\worstcase}(\cM; V_n)$ and $\cE^{\averagecase}_\rho(\cM; V_n)$. We also display the trend of the upper bound of the Kolmogorov $n$-width $d_n^\worstcase(\cM)$ predicted by \cref{thm:kolmo-width-poisson}.}
    \label{fig:forward_decay}
\end{figure}

In \cref{fig:forward_decay} we can see this behaviour in practice.
The form in which we have formulated the weak greedy algorithm assumes that, for each $\theta\in \Theta$, we approximate $u(\theta)$ with the best approximation $P_{V_n}(u(\theta))$. This is not true since in practice we compute the element $\hat u_n(\theta)\in V_n$ which is the solution to the Galerkin projection from \eqref{eq:vf-galerkin} when we use $V_n$ as the finite dimensional space. Consequently, the weak greedy algorithm reads
\begin{equation}
    \theta_1 \in \{ \theta \in \Theta\cond \Vert \hat u_n(\theta) \Vert_V \geq \gamma \sup_{\theta\in \Theta} \Vert \hat u_n(\theta) \Vert_V  \},
    \quad \text{and} \quad
    u_1 \coloneqq \hat u_n(\theta_1)
\end{equation}
and for $n>1$,
\begin{equation*}
    \theta_n \in \{ \theta \in \Theta \cond \Vert u(\theta)-\hat u_{n-1}(\theta) \Vert_V \geq \gamma \sup_{\theta\in \Theta} \Vert u(\theta)-\hat u_{n-1}(\theta) \Vert_V  \},
    \quad \text{and} \quad
    u_n \coloneqq \hat u_n(\theta_n)
\end{equation*}
For the Poisson problem \eqref{eq:vf}, errors of the type $\Vert u(\theta)-\hat u_{n-1}(\theta) \Vert_V$ are equivalent to residual computations which can be computed in an efficient way. Depending on the nature of the PDE, this replacement may be more or less difficult to do. We will not dive further into this topic for the sake of brevity.

\corr{As mentioned before, in practice, the optimization over $\cM$ in \eqref{eq:vanilla-greedy-n1} and \eqref{eq:vanilla-greedy-n} is done over the finite training set $\cM^K$. In principle, the paradigm of the weak greedy algorithm can account for the inaccuracies introduced by working on a finite amount of samples (the parameter $\gamma$ of the algorithm degrades and becomes smaller). However, to rigorously relate the amount of samples to $\gamma$, in principle one needs training sets that are an $\eps$-covering of $\cM$ (for $\eps>0$ taken as small as the final accuracy of the application requires, see e.g.~\cite{MMT2016} for a discussion about this). Depending on $\cM$, the amount of samples may become quickly very large which is also connected to the so-called covering numbers. Working with sets that are not $\eps$-coverings is very challenging, and very little is understood in this respect. We refer to \cite{CDDN2020} for a work in which this aspect is studied.}


\subsection{Nonlinear Strategies}
\label{sec:fwd:nonlinear}
We say that a strategy is nonlinear when the approximation of $\cM$ is no longer done with linear spaces $V_n\subset V$. We first motivate why such approaches are needed, and then we briefly explain some strategies that are currently being investigated in the literature. \corr{The topic of nonlinear model order reduction is currently attracting a lot of attention, and numerous different ideas are being explored. We consider that this is an advanced topic which goes a bit beyond the purpose of this introductory tutorial. Therefore we only provide the main ideas, and give a few examples of existing strategies. We refer to \cite{HPU2026} for a recent extensive review on this topic.}

\subsubsection{Why do we need nonlinear Model Order Reduction?}

As we have seen in the previous sections, linear reduced models can bring a significant reduction of complexity for elliptic problems thanks to the fact that the Kolomogorov $n$-width decays exponentially fast (see \cref{thm:kolmo-width-poisson}). The situation is quite different for problems transporting discontinuities that vary together with the parameters. This typically occurs for hyperbolic PDEs when parameters influence the transport velocity and therefore the positions of shocks. To illustrate this, we recall a toy example given in \cite{BCMN2018}. We consider the univariate transport equation
\begin{equation}
    \partial_t u(\theta)+\theta \partial_x u(\theta)=0, \quad x\in\R, \; t\geq 0,
    \label{eq:transport}
\end{equation}
with the characteristic function as initial condition $u_0(x)=\indicative_{]-\infty,0]}$, and the velocity as the parameter of interest $\theta\in [0,1]$. We consider the parametrized family of
solutions at time $t=1$ restricted to $x\in [0,1]$, that is
\begin{equation}
    \label{eq:manifold-indicative-fun}
    \cM = \{ u(\theta)(x)=\indicative_{[0,\theta]}\; : \; \theta\in [0,1]\}.
\end{equation}
In \cref{fig:transport} we can see different profiles for $u(\theta)(x)$ for various values of $\theta$.

Here, we choose to work in the Hilbert space $V=L^2([0,1])$ which contains such discontinuous functions.
In this context, the covariance operator $C$ (from \cref{thm:svd-ppdes}) has its spectral decomposition given by the Karhunen–Loève orthonormal basis of eigenfunctions $f_k$ which, in this case, correspond to the Fourier series associated to the interval $[0, 1]$, that is $f_k^s=2\sin(2k\pi x)$ and $f_k^c=2\cos(2k\pi x)$. The associated eigenvalues, due to the effect of the discontinuity, decrease proportional to $k^{-1}$ thus even if we build our reduced basis with the first $n$ Fourier terms (which correspond to $V^{\text{opt}}_n$) we get an approximation error proportional to $n^{-1/2}$
\begin{equation}
    \sup_{u\in\cM} \|u-P_{V_n}u\| \sim n^{-1/2}.
\end{equation}
which also implies that $d_n^{\averagecase}(\cM) \sim n^{-1/2}$.

On the other hand, a linear reduced space generated by the greedy algorithm is of the form
\begin{equation}
    V_n = \vspan\{ \indicative_{[0,\theta_i]} \; : \, i=1,\ldots, n \}
\end{equation}
for some points $\theta_1,\dots,\theta_n\in [0,1]$, and therefore it is equal to the space of piecewise constant functions on the intervals $[\theta_i,\theta_{i+1}]$,
assuming that these points have been increasingly ordered.
By taking a $\theta$ to be the midpoint of the largest of such intervals, that has length larger than $n^{-1}$, we can check that
\begin{equation}
    \|\indicative_{[0,\theta]}-P_{V_n}\indicative_{[0,\theta]}\| \geq  \frac 1 2 n^{-1/2},
\end{equation}
and therefore the approximation rate is not better than with the Fourier space. More generally it can be checked that the Kolmogorov
$n$-width of $\cM$ in $V$ decays like $n^{-1/2}$, that is, any linear approximation method cannot have a better rate (for a complete proof, see \cite[Chapter 3, see equation (3.76)]{BCOW2017}). This rate is too slow to make linear reduced models effective in this case (remember that we want to work with $n$ small). By now, similar results have been obtained for different manifolds of hyperbolic PDEs (see \cite{Welper2020transformed, OR2016}). We also refer to \cite{GK2019} for the same type of result for wave equations. All these results point to the fact that one should in general expect that the convergence rate of the Kolmogorov $n$-width is slow \corr{as soon as the solution has discontinuities or steep gradients}. So we need to shift to nonlinear approximation spaces for model reduction for these PDEs.

\begin{figure}[ht]
    \centering
    \includegraphics[width=0.4\linewidth]{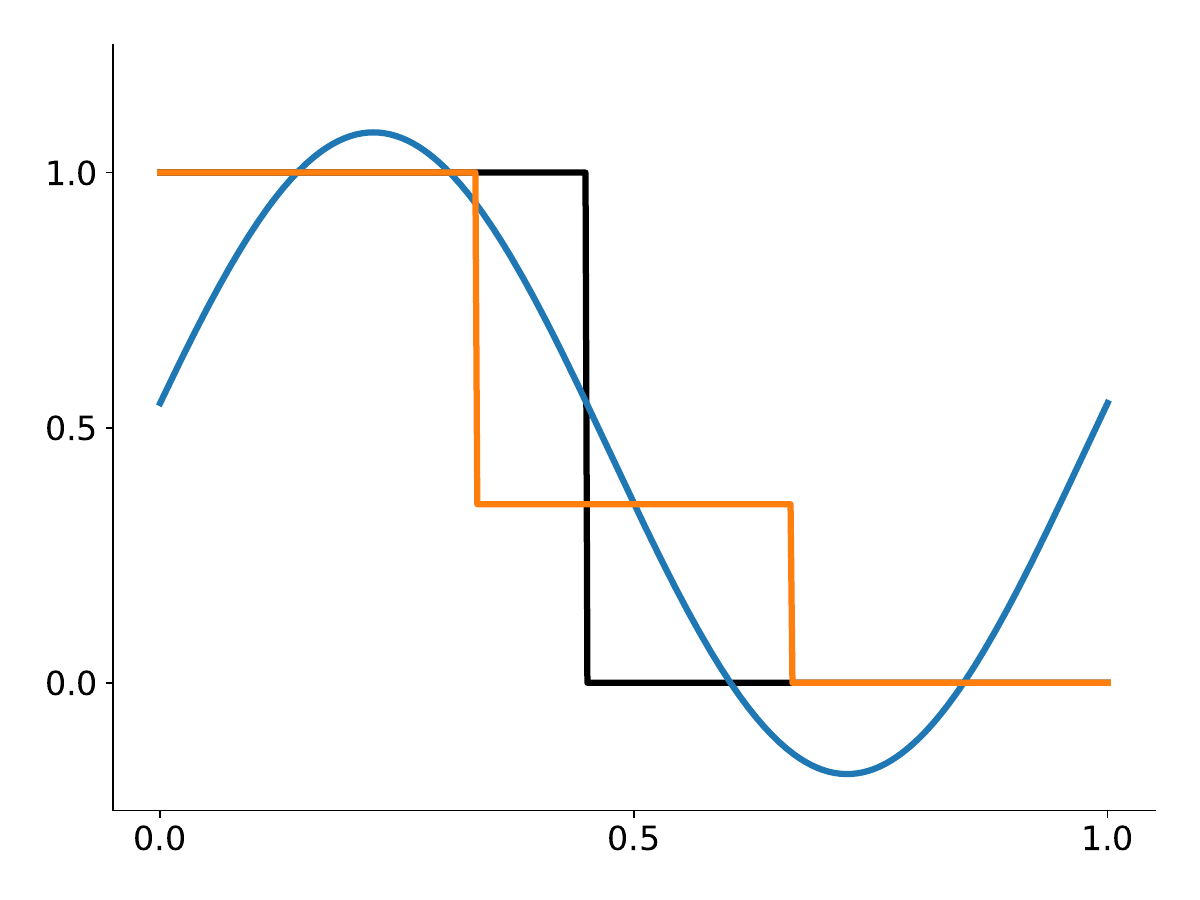}
    \includegraphics[width=0.4\linewidth]{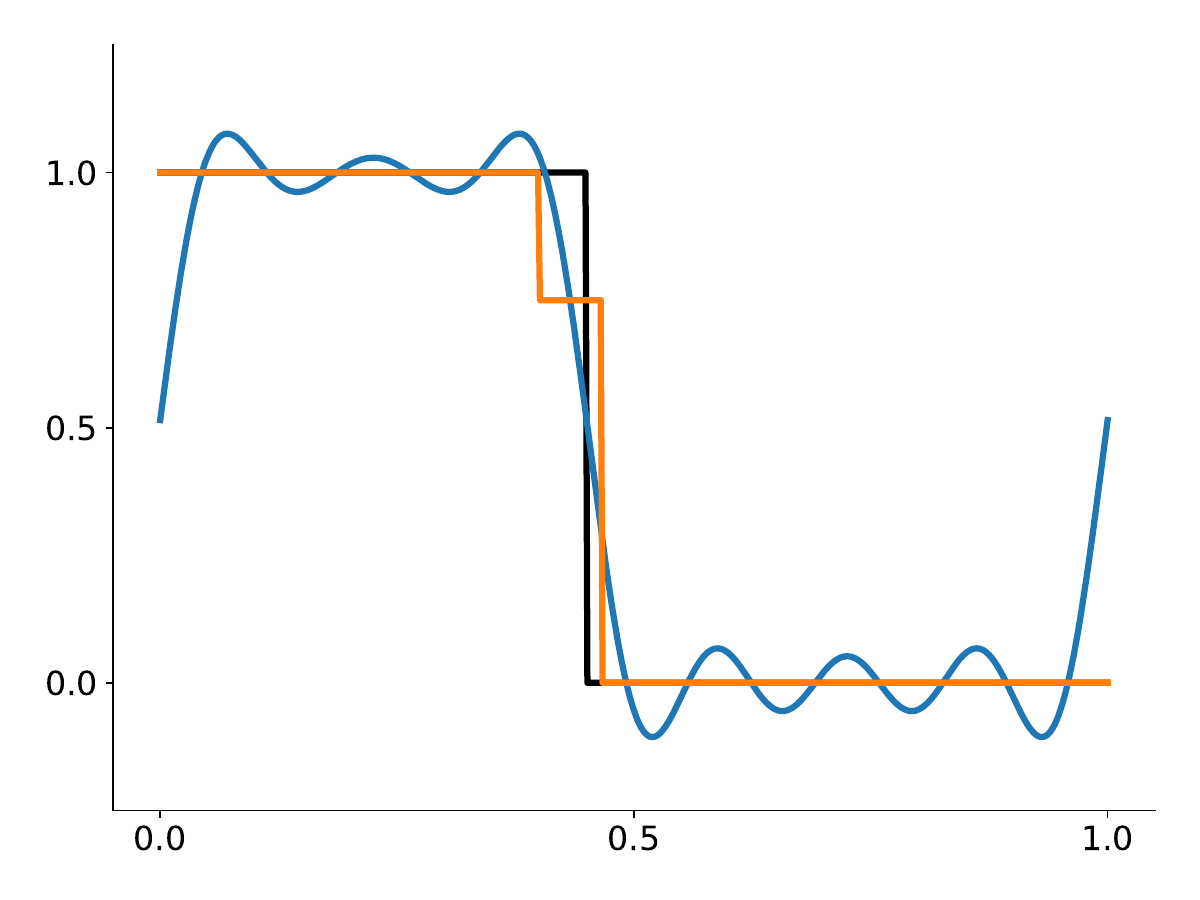}
    \includegraphics[width=0.4\linewidth]{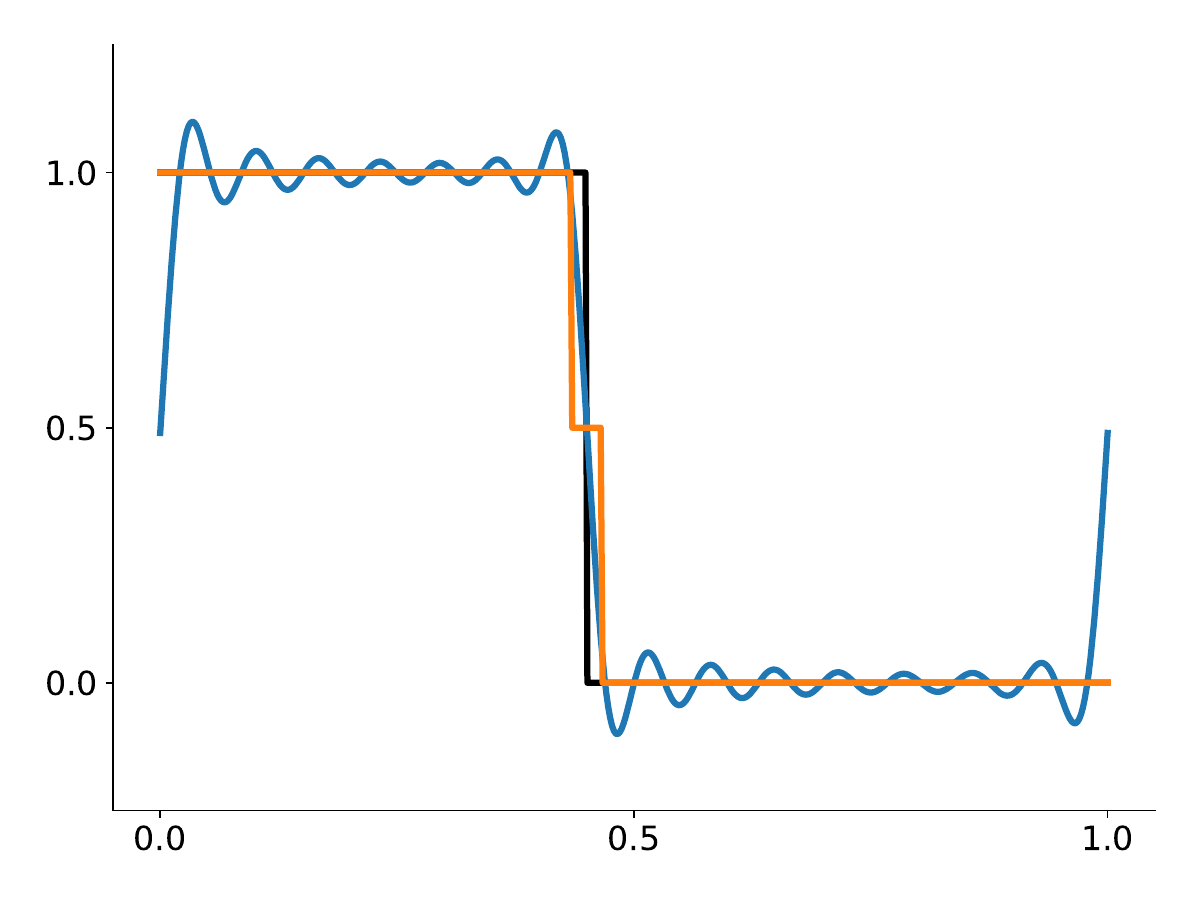}
    \caption{\new{Some MOR methods applied to the manifold $\cM$ of \cref{eq:manifold-indicative-fun} generated by translations of Heaviside functions.} We see in black the profile of $u(\theta)$ which corresponds to the final state ($t=1$) of \eqref{eq:transport} for $\theta=0.45$. We also show the best approximations $P_{V_n}\indicative_{[0,\theta]}$ using the spaces $V_n$ obtained by the greedy method (orange), or the SVD (blue). The dimension changes from $n=3$ (top-left), $n=15$ (top-right) and $n=30$ (bottom). We observe the presence of large oscillations for all methods.}
    \label{fig:transport}
\end{figure}

\subsubsection{Piecewise Affine Approximation and Library Approximation}
\label{sec:fwd:pwa}
If we think how to go beyond a simple linear space, one first idea is to work with piecewise linear approximations. For this, we need to build a partition of the parameter space $\Theta$ into $K$ disjoint subdomains
$$
\Theta=\cup_{i=1}^K \Theta_i, \quad \text{with }\Theta_i\subseteq\Theta \text{, and }\Theta_i\cap \Theta_j = \emptyset \text{ if $i\neq j$},\; \forall (i,j)\in\{1,\dots, K\}^2.
$$
This partition induces a partition of the manifold $\cM$,
$$
\cM = \cup_{i=1}^K \cM_i, \quad \cM_i\coloneqq u(\Theta_i),\; \text{ and } \cM_i \cap\cM_j= \emptyset \text{ if $i\neq j$}.
$$
For each $i$, we approximate $\cM_i$ with a linear space $V_{n_i}^i$ of dimension $n_i$. Denoting $\cL\coloneqq \{V_{n_1}^1,\dots, V_{n_K}^K\}$, the piece-wise parameter-to-solution mapping is of the form
\begin{align*}
    \label{eq:widehat-u-pw}
    \widehat u_n^{\pw}:\Theta \to \cL.
\end{align*}
In our Poisson model problem, if we want to approximate $u(\theta)$ for a given $\theta\in \Theta_i$, we solve the Galerkin problem \eqref{eq:vf-galerkin} with the subspace $V_{n_i}^i$.

The main question then is how to build the partition and what is the quality criterion to judge whether one is better than another. We refer to \cite{GM2024} for a discussion on this matter, and also for an overview of various strategies, and how difficult they are to implement. In essence, the piece-wise affine approach is easy to use in place when $\Theta$ is a Cartesian product domain. We can then develop a partition strategy based on dyadic splittings (we refer to
\cite{EPR2010,MS2013,BCDGJP2021, DCAKR2022, GJ2024} for works in this direction). The construction is not that straightforward as soon as we want to consider either generic shapes for each partition, or very general domains which are not necessarily Cartesian products. In \cite{GM2024}, an approach for general domains was introduced, and the way to store the information is based on a tree structure of the different subspaces $V_n^k$ that are built.

The idea of working with several subspaces belongs to a broader class of approximation methods called library or dictionary approximation. We again refer to \cite{GM2024} for a description of the main concepts.

\subsubsection{\corr{Beyond piece-wise affine strategies: nonlinear parametric approximations}}
The main idea to go beyond piece-wise affine methods relies on parametric mappings (also called decoder mappings) of the form
\begin{equation}
    \label{eq:decoder}
    \varphi: W \to V,
\end{equation}
where $W$ is a space of parameters that are amenable to manipulate in practice. These parameters are not to be confused with the parameters $\theta\in \Theta$ from the PDE. Thus, to avoid confusion in the terminology, we will refer to the elements $w\in W$ as the weights, and to $W$ as the space of weights. In full generality, $W$ is a Riemannian manifold of dimension $n$ however the most common choice is just the flat space $W=\bR^n$. The image set
\begin{equation}
    V_n \coloneqq \varphi(W)=\{ \varphi(w) \cond w \in W \} \subseteq V
\end{equation}
is a subset of $V$ parametrized by the elements $w\in W$. In general $V_n$ is not going to be a linear subspace of $V$ as in our earlier discussion unless we make specific choices on $\varphi$. However, we overload the notation in this part of the text because we are building a generalization of the previous construction in which the case of $V_n$ being a linear subspace happens for particular choices of $\varphi$.

Most common approximation families can be expressed as a parametric mapping. Here are some examples:
\begin{itemize}
    \item \textbf{Polynomial-type approximations:} The first important class of parametric approximations are those of polynomial type. Choosing $W=\bR^n$, we consider a polynomial ansatz
    \begin{equation}
        \label{eq:poly decoder}
        \varphi^{\poly}(w) = \sum_{\nu\in E } w^\nu v_\nu, \quad \forall w=(w_1,\dots,w_n)^\top \in \bR^n,
    \end{equation}
    where $\nu = (\nu_1,\dots, \nu_n)\in \bN_0^n$ is a multi-index of order $\abs{\nu}=\nu_1+\dots+\nu_n$ and
    $$
    w^\nu =w_1^{\nu_1}\dots w_n^{\nu_n},
    $$
    and the set $E$ is a finite set of multi-indices of $\bN_0^n$. The index set $E$ and the functions $v_\nu \in V$ are chosen a priori. This construction generates the approximation space
    $$
    V_n^{\poly} = \varphi^{\poly}(W) = \set{\varphi^{\poly}(w) \cond w\in W },
    $$
    which is of dimension $\dim(V_n^{\poly})=n$. We remark that this approximation class includes the important case of linear combinations of functions, since we can choose $E = \set{e_i}_{i=1}^n$, where the $e_i$ are the canonical vectors of $\bR^n$. With this choice, we obtain
    \begin{equation}
        \label{eq:poly decoder lin}
        \varphi^{\lin}(w) = \sum_{i=1}^n w_i v_i,
    \end{equation}
    Observe that choosing the functions $v_i$ as hat functions allows us to express finite-element-type approximations in the current framework. Also, we can express the linear approximation with the Greedy basis by choosing the $v_i$'s as the solutions $u(\theta_i)$ selected by the greedy algorithm (and similarly for the SVD basis).
    \item \textbf{Neural networks:} An important instance of a fully nonlinear parametric mapping are neural networks. The most basic instance of this family are Shallow Neural Networks (SNN) which are mappings of the form
    \begin{equation}
        \label{eq:SNN decoder}
        \varphi^{\SNN}(w) = \sum_{i=1}^k c_i \sigma(a_i \cdot + b_i),
    \end{equation}
    where $w = \prt*{(a_i,b_i,c_i)}_{i=1}^p$ and \corr{$(a_i,b_i,c_i)\in \bR^d\times \bR \times \bR$}. Consequently $W= \bR^{n}$ with \corr{$n=k(d+2)$}. The function \corr{$\sigma:\bR \to \bR$ is the so-called activation function (typical choices are the ReLU function, sigmoid, hyperbolic tangent, etc.). $\sigma$ needs to be nonlinear to guarantee density properties as $k\to \infty$ with respect to classical ambient Hilbert spaces $V$ encountered in PDE theory.}.
    This construction generates the fully nonlinear approximation set
    $$
    V_n^{\SNN} = \varphi^{\SNN}(W) = \set{\varphi^{\SNN}(w) \cond w\in W }.
    $$
\end{itemize}
Once a parametric mapping $\varphi$ is chosen, the main idea to do model order reduction is to consider that the weights $w\in W$ depend on the physical PDE parameters $\theta\in \Theta$. In this way, one approximates the exact solution $u(\theta)$ with $\varphi(w(\theta))$. Therefore, here we have
\begin{align}
    \hat u_n: \Theta & \to V                                         \\
    \theta           & \mapsto \hat u_n(\theta) = \varphi(w(\theta))
\end{align}
The main question then becomes how to build numerically efficient strategies to map from a physical parameter $\theta$ to an appropriate weight $w\in W$. For this, one key idea is to resort once again to nonlinear parametric approximations to represent the parameter-to-weight mapping, and to use supervised learning techniques to train the weights.

Fully nonlinear parametric mappings are the subject of very active research, and numerous
strategies are currently being proposed and studied. For an extensive overview on all currently existing techniques, we refer to \cite{HPU2026}. Below, we only give two examples of such strategies:

\begin{itemize}
    \item \textbf{Nonlinear Compressive ROM:}
    As we have seen in \cref{sec:fwd:linear}, quasi-optimal linear spaces $V_n=\vspan\{v_i\}_{i=1}^n$ are built either by SVD (in the average case), or by a greedy algorithm (in the worst case). To enhance the accuracy of these spaces, here is a strategy proposed in \cite{CFSM2023}. If we accept to pay the price of computing the SVD or the greedy algorithm up to a larger dimension $N>n$, we can then consider an ansatz of the form
    $$
    \widehat u_n(\theta) = \sum_{i=1}^n \new{w_i}(\theta) v_i + \sum_{i>n}^N \psi_i(\new{w_1(\theta),\dots, w_n(\theta)}) v_{i}.
    $$
    In the case of the Poisson example, we can obtain the weights $w_i(\theta)$ by solving the Galerkin problem \eqref{eq:vf-galerkin} with the subspace $V_n$ of small dimension.
    We can then use a supervised learning algorithm to train the mappings $\psi_i:\bR^n\to \bR$ for $i=n+1,\dots,N$. These mappings can be searched in an approximation class generated by polynomials, neural networks or any other type of parametric function family (see \cref{fig:transport_nonlinear} for an example).

    \begin{figure}[ht]
        \centering
        \includegraphics[width=0.6\linewidth]{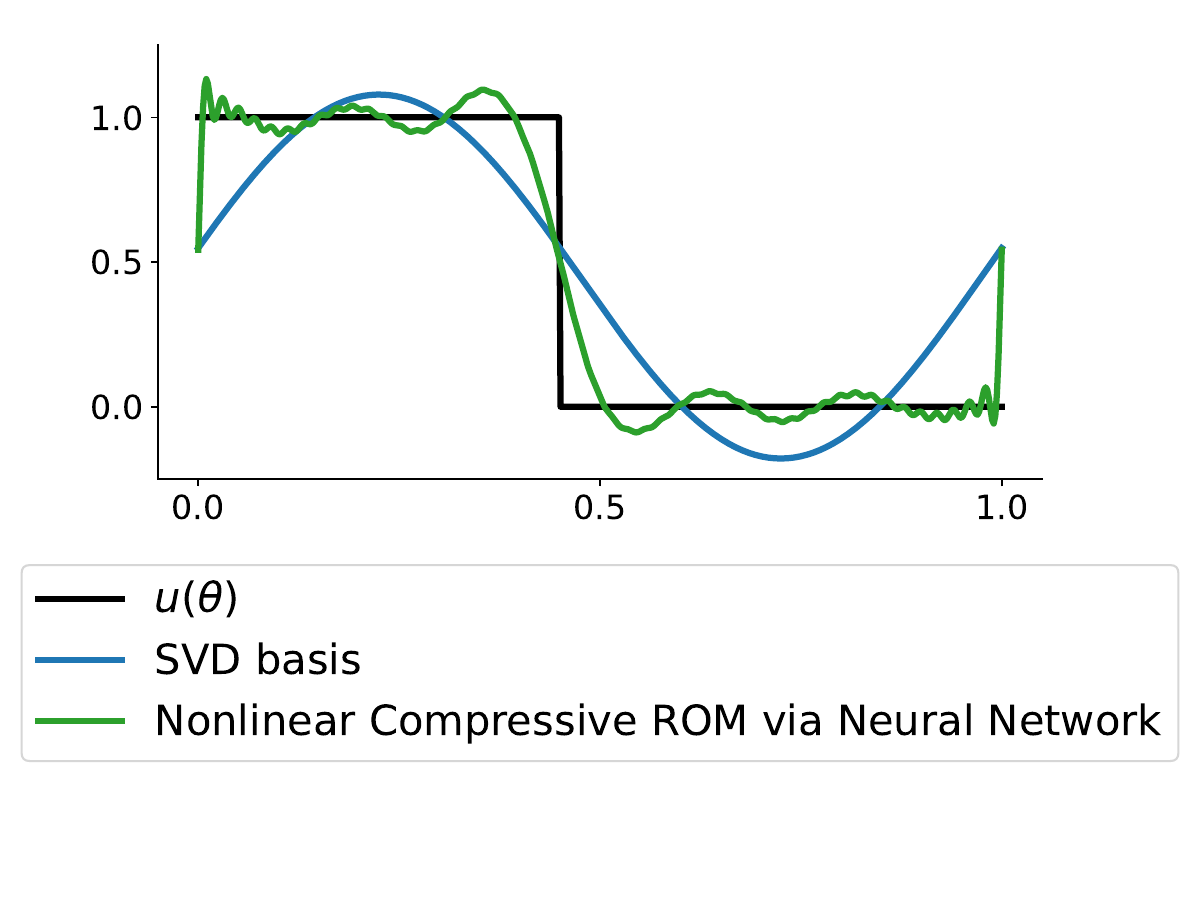}
        \caption{Black line: the function $u(\theta)$ which corresponds to the final state ($t=1$) of \eqref{eq:transport} for $\theta=0.45$. Blue line: the best approximation obtainable by an SVD reduced basis of dimension $n=3$. Green line: a nonlinear method as in \cite{CFSM2023} using a neural network to learn the mappings $\psi_i$.}
        \label{fig:transport_nonlinear}
    \end{figure}

    \item \textbf{Quadratic Manifold:}
    \label{sec:fwd:quadratic}
    A somewhat complementary approach to the nonlinear compressive strategy is the so-called quadratic manifold approach. This idea finds its roots in works from structural mechanics such as \cite{JTRR2017}. It was then taken as a starting point in \cite{GWW2023} (and many subsequent works such as \cite{GBW2023}) to develop strategies revolving around the ansatz
    $$
    \widehat u_n(\theta) = \sum_{i=1}^n w_i(\theta) v_i + \sum_{i,j=1}^n w_i(\theta) w_j(\theta) v_{i,j}.
    $$
    Like in the nonlinear compressive strategy, the spirit is to enhance the accuracy of the linear space $V_n=\vspan\{v_i\}_{i=1}^n$ obtained by SVD or with the greedy algorithm. For this, in the quadratic manifold approach we search for the best functions $v_{i,j}\in V$ such that an expansion involving quadratic terms $w_i(\theta)w_j(\theta)$ gives a good accuracy. If we find such $v_{i,j}$, then the computation boils down to estimating only the $n$ terms $w_i(\theta)$ for $i=1,\dots,n.$
\end{itemize}

\subsubsection{\corr{Beyond piece-wise affine strategies: transformations}}
\label{sec:fwd:fully-nonlinear}

In addition to fully nonlinear parametric approximations, another powerful paradigm in nonlinear model reduction is based on the use of transformations. The underlying idea is that, although the solution manifold associated with a parametric PDE may be poorly approximated by linear spaces, it may become much more amenable to reduction after a suitable change of variables. More precisely, one seeks an approximation of the form
$$
u(\theta) \approx T\bigl(\hat u_n(\theta)\bigr),
$$
where $T:V\to V$ is a transformation and $\hat u_n(\theta)$ approximates a transformed representation of the solution. The transformation $T$ may be prescribed from prior knowledge of the governing equations or inferred directly from data. Its purpose is to factor out the dominant nonlinear effects such as transport, moving interfaces, or geometric deformations. This way, it is expected that the transformed solution manifold
\corr{
    $$
    T^{-1}(\cM)
    :=
    \prt*{
        T^{-1}(u(\theta))
        :
        \theta\in\Theta
    },
    $$
}
has significantly smaller Kolmogorov $n$-widths than the original solution manifold $\cM$. Linear reduction techniques, or closely related approximation methods, are then applied to the transformed problem. Several important instances of this idea are described below.

\begin{itemize}
    \item \textbf{Physics-informed transformations:}
    A natural way of constructing the transformation $T$ is to exploit prior knowledge about the underlying dynamics. Typical examples arise from group actions such as translations, rotations, or other symmetries of the problem. In transport-dominated PDEs, this idea has been used to capture advective phenomena by choosing $T$ according to the transport dynamics generated by the governing equations (see, e.g., \cite{OR2013}). More sophisticated variants combine linear approximation techniques with nonlinear transport maps evolving along characteristic curves, thereby factoring out the dominant transport behavior before applying model reduction (see, e.g., \cite{Welper2017-TSI, Welper2020transformed, RPM2023, RW2026}).

    \item \textbf{Data-driven registration and alignment methods:}
    A closely related family of approaches consists of inferring the transformation $T$ directly from data. These methods are often equation-agnostic and are inspired by image registration techniques. Their objective is to align coherent structures, interfaces, shocks, or other dominant features across parameter values or time instances, thereby transforming the solution manifold into a representation that is more amenable to linear approximation. In this setting, the transformation is not prescribed from the governing equations but learned from the snapshots themselves (see \cite{Taddei2020, FTZ2022, Blickhan2024}).

    \item \textbf{Spaces of measures and general metric spaces:}
    Nonlinear approximation methods also arise naturally when the ambient solution space itself lacks a linear structure. A prominent example is provided by gradient flows in Wasserstein spaces, where linear combinations of solutions are generally not meaningful. In such settings, model reduction must inherently rely on nonlinear constructions. One successful strategy consists in representing $\hat u_n(\theta)$ as a Wasserstein barycenter of suitably selected snapshots. In one spatial dimension, this can be achieved through a physics-informed transformation based on inverse cumulative distribution functions, which maps the problem into a space where convex combinations become meaningful (see \cite{ELMV2020, BBEELM2023}). Extensions to higher-dimensional Wasserstein spaces have also been developed, although they no longer rely on such a transformation-based linearization (see, e.g., \cite{DDE2023, DFM2025}). Another related approach consists of constructing reduced models through the approximation of push-forward maps in combination with registration approaches, as proposed in \cite{Blickhan2024}. We also refer to \cite{IL2014} for an early work leveraging ideas from Optimal Transport to develop physics-informed transformations.
\end{itemize}

\subsection{Further reading}

There are by now several books about forward ROM which cover particularly well the case of linear approximation, see \cite{HRS2015, QMN2016, BCOW2017}. These references also explain how to apply ROM to problems such as uncertainty quantification and optimal control. There are plenty of more advanced concepts that we have not covered in our discussion, and which are sometimes not fully covered either in these books. Some relevant ones are:
\begin{itemize}
    \item How to reliably, and efficiently work with residuals in the weak greedy algorithm This question is very related to a good choice of a weak formulation of the problem at hand. In presence of nonlinearities, we need to involve further steps, usually known under the name of hyperreduction methods. The Empirical Interpolation Method is probably the most well-known approach, see \cite{BMNP2004}.
    \item In time-dependent problems, it is often interesting to consider that the manifold set of solutions depends on time. We would thus have to work with $\cM(t)$, and in this case one has to compute reduced models $V_n(t)$ that depend on time. One challenge is about how to make the basis evolve in time, and there is a lot of activity in this direction. We refer to \cite[Section 4.3]{HPR2022} for an overview. \corr{Another relevant approach is the so-called POD-Greedy method, see \cite{haasdonk2013convergence}.}
    \item Many problems come with important structural properties. Examples are conservation laws, symmetries, symplecticity, reversibility and invariants of motion. A topic of vibrant research at the moment is related to building reduced models $V_n$ that preserve such properties, see \cite[Section 4.2]{HPR2022} or \cite{ELMV2020}.
\end{itemize}
\new{Concerning nonlinear model reduction methods, we refer to \cite{HPU2026} for an extensive overview of current research.} Last but not least, for training on how to implement advanced ROM concepts, we refer to \cite{RBSP2024} which is a book that offers an extensive collection of worked out implementation problems with increasing complexity from the field of computational mechanics.


\section{Inverse State Estimation}
\label{sec:inverse}
\emph{Inverse Problems} occur when the parameter $\theta$ is not given, and, instead, we only observe a vector of measurements
\begin{equation*}
    z=(z_1,\dots,z_m)\in \bR^m, \quad z_i=\ell_i(u),\quad i=1,\dots,m,
\end{equation*}
where each $\ell_i\in V'$ is a known continuous linear functional on $V$, and we assume that the set of $\{\ell_i\}_{i=1}^m$ are linearly independent. We will sometimes use the notation
\begin{equation*}
    z=\ell(u), \quad \ell=(\ell_1,\dots,\ell_m).
\end{equation*}
There are two main types of inverse problems:
\begin{enumerate}
    \item
    \emph{State estimation:} The task is to recover an approximation $u^*$ of {the state} $u\in V$ from the observation {$z=\ell(u)$}. Since we assume that $V$ is high or infinite dimensional, there are in general many $u$ that satisfy $\ell(u)=z$ so we decrease this ``uncertainty'' by adding extra assumptions. In our case, we assume that $u$ belongs to the manifold $\cM$.

    State estimation is a linear inverse problem in the sense that the forward map $\ell:u\mapsto \ell(u)$ is linear. It is however challenging because the target $u$ lives in an infinite dimensional space $V$. The additional assumption that $u\in \cM$ eases the problem only to some extent. On the one hand, it means that $u$ is in a manifold that can be parametrized with $p$ parameters. Since we receive $m$ observations, recovering $u\in \cM$ from $z$ is an ill-posed problem as soon as $p>m$. This ill-posed setting is the one that we are often implicitly considering since it corresponds to the situation in which the physical model is complex ($p$ is large), and we only observe it with a few sensors ($m$ is moderate). On the other hand, the membership of $u$ to $\cM$ makes it difficult to recover $u$ even when the problem is not ill-posed. This is because $\cM$ has a complicated geometry, which is only partially known to us by solving forward problems $\theta\mapsto u(\theta)$ for different values of $\theta\in\Theta$.
    \item
    \emph{Parameter estimation:} Here the goal is to recover an approximation $\theta^*$ of the parameter $\theta$ from the observation {$z = \ell(u)$ when $u=u(\theta)$}. This is a nonlinear inverse problem, for which the prior information available on $\theta$ is given by the domain $\Theta$.
\end{enumerate}

In the above description, we have assumed that the observations are noiseless, that the $\ell_i$ are perfectly known, and that the physical model is perfect so that we have $u\in \cM$. Of course, in general, none of these three assumptions are true:
\begin{itemize}
    \item Measurements are usually affected by noise. To account for this, we need to add an extra modeling assumptions. In general, one assumes additive noise
    \begin{equation*}
        z=\ell(u)+\eta,
    \end{equation*}
    and $\eta\in \bR^m$ is a vector of random variables that follows a certain probability distribution. One then needs to decide which probability distribution this is, and whether the sensors' noise is correlated or uncorrelated.
    \item The sensors' response is usually not completely known which means in turn that the functionals $\ell_i$ do not fully describe the measurment process. One can account for this lack of knowledge by adding an extra additive noise term.
    \item The physical model in general does not fully capture the behaviour of the real system so the true observed state is an element $u\in V$, and it may deviate from the solution manifold $\cM$ by a considerable amount.
\end{itemize}
In the following, we place ourselves in the idealized setting without modeling errors and measurement noise. We do this for clarity of exposition, but we emphasize that one can easily incorporate these inaccuracies in the theory and the algorithms that we present below.


\subsection{Optimal Algorithms for State Estimation}
Since we are assuming that we have no noise, our observations are of the form
\begin{equation}
    \label{eq:observations}
    z_i = \ell_i(u), \quad i=1,\dots,m.
\end{equation}
\new{We denote} by $\omega_i\in V$ the Riesz
representers of the $\ell_i$, such that $\ell_i(v)=\inner{\omega_i,v}$ for all $v\in V$ (see \cref{fig:riesz}), and we define the observation space
\begin{equation*}
    \Wm:={\rm span}\{\omega_1,\dots,\omega_m\}\,\subset V.
\end{equation*}
\new{To avoid any possible confusion, we emphasize that in this section $\Wm$ has a totally different meaning with respect to the space of weights $W$ from \cref{eq:decoder} for parametric approximations. Using the observation space $\Wm$,} the measurement data $z = (z_1,\dots, z_m)^T$ can equivalently be represented by
\begin{equation*}
    \omega=P_\Wm u,
\end{equation*}
where $P_\Wm$ is the orthogonal projection from $V$ onto $\Wm$ (see \cref{fig:riesz_projection}). This equivalence comes from the fact that we can write
$$
P_\Wm u = \sum_{i=1}^n c_i \omega_i,
$$
for some coefficients $c_i\in \bR$. Knowing the measurement data allows us to write
$$
z_i = \ell_i(u) = \left< \omega_i, u\right> = \left< \omega_i, P_\Wm u\right> = \sum_{j=1}^m c_j \left< \omega_i, \omega_j \right>,\quad \forall i=1,\dots, m.
$$
Thus the vector of coefficients $c =(c_1,\dots,c_m)^T$ is the unique solution to the linear system
$$
\texttt{B} \,c = z
$$
where
$$
\texttt{B} \coloneqq ( \left< \omega_i, \omega_j \right>)_{1\leq i, j \leq m}
$$
is an invertible matrix because the $\{ \omega_i \}_{i=1}^m$ are linearly independent (because their associated linear functionals $\{ \ell_i \}_{i=1}^m$ are assumed to be independent). Therefore knowing $z$ is equivalent to knowing $c$, and also $w =P_\Wm u$.

\begin{remark}
    Given a linear functional $\ell_i\in V'$, its Riesz representer $\omega_i$ is defined by the variational problem:
    \begin{center}
        Find $\omega_i\in V$ s.t.~$\inner{\omega_i, v} = \ell_i(v),\;\forall v\in V.$
    \end{center}
    In practice, one can find an approximation of $\omega_i$ by Galekin projection. We choose a finite dimensional space $V_h$ (usually, a finite element space), and then the problem is:
    \begin{center}
        Find $\omega^{(h)}_i\in V_h$ s.t.~$\inner{\omega^{(h)}_i, v_h} = \ell_i(v_h),\;\forall v_h\in V_h.$
    \end{center}
    For example, if $V=L^2(\Omega)$, and if $\ell_i(v)\coloneqq \int v(x)e^{-\frac{(x-x_i)^2}{2\sigma^2}} \mathrm{d} x$, then we can directly identify $\omega_i = e^{-\frac{(x-x_i)^2}{2\sigma^2}}$. If $V=H^1_0(\Omega)$, then we have to consider the variational problem
    \begin{center}
        Find $\omega_i\in H^1_0(\Omega)$ s.t.~$\int_\Omega \nabla \omega_i \cdot \nabla v = \int v(x)e^{-\frac{(x-x_i)^2}{2\sigma^2}} \mathrm{d} x,\;\forall v\in H^1_0(\Omega).$
    \end{center}
    \cref{fig:riesz} gives a visual illustration of this example.
\end{remark}

\begin{figure}[ht]
    \centering
    \includegraphics[width=0.2\linewidth]{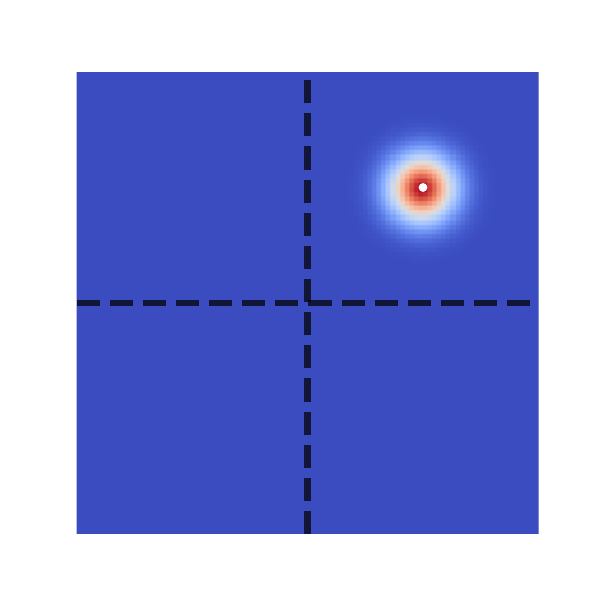}
    \includegraphics[width=0.2\linewidth]{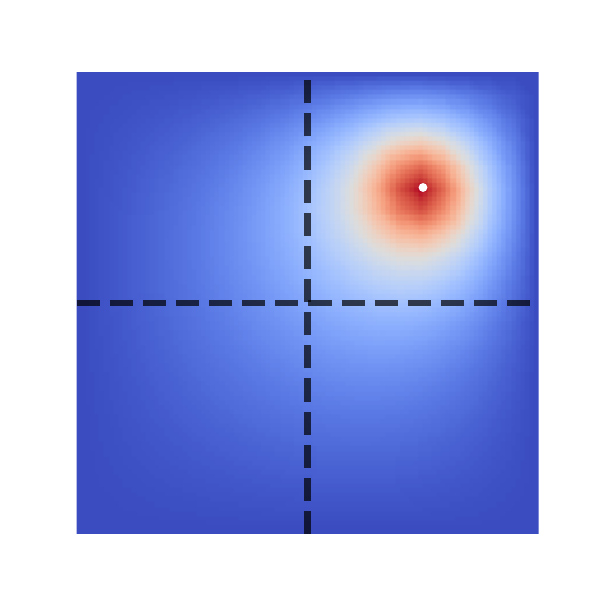}\\
    \includegraphics[width=0.2\linewidth]{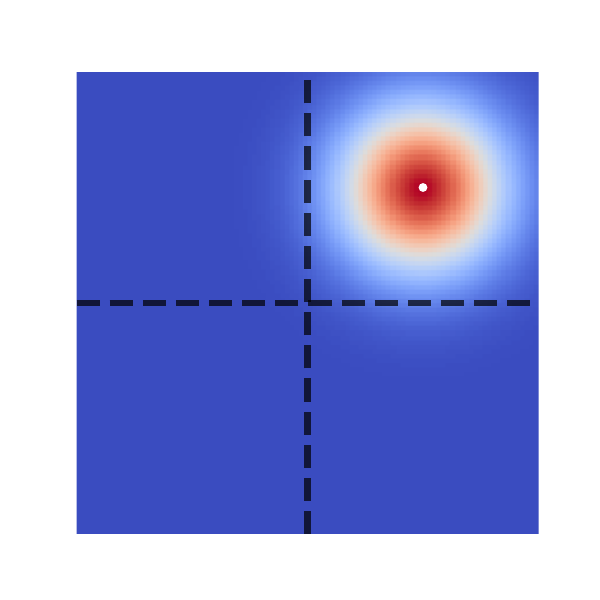}
    \includegraphics[width=0.2\linewidth]{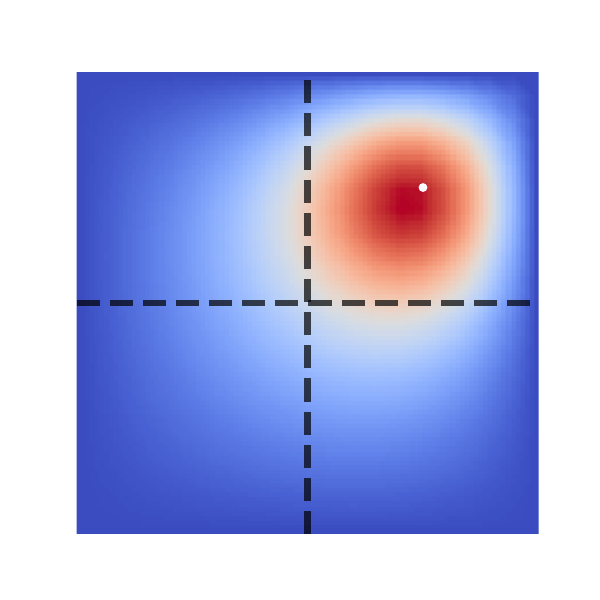}
    \caption{Riesz representer $\omega_i$ if we work with $\ell_i(v)=\int v(x)e^{-\frac{(x-x_i)^2}{2\sigma^2}} \mathrm{d} x$ with $x_i=(0.5,0.5)$. Images on the left: $V=L^2([-1,1]^2)$. On the right: $V=H_0^1([-1,1]^2)$. $\sigma=0.1$ (up) and $\sigma=0.25$ (down).}
    \label{fig:riesz}
\end{figure}

\begin{figure}[ht]
    \centering
    \includegraphics[width=0.3\linewidth]{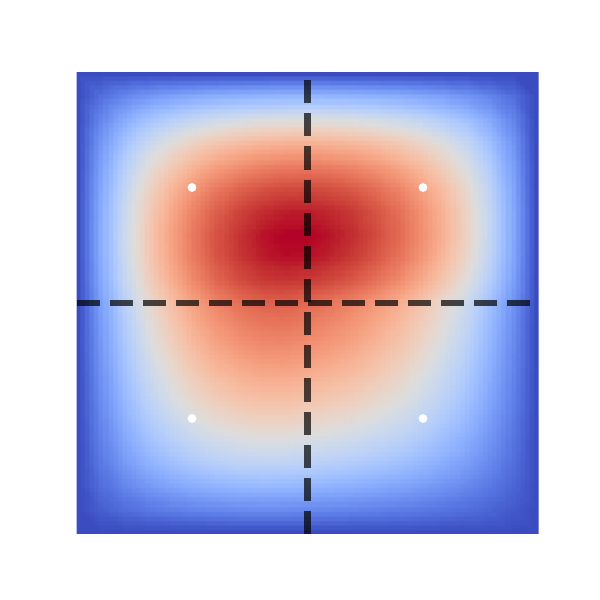}
    \includegraphics[width=0.3\linewidth]{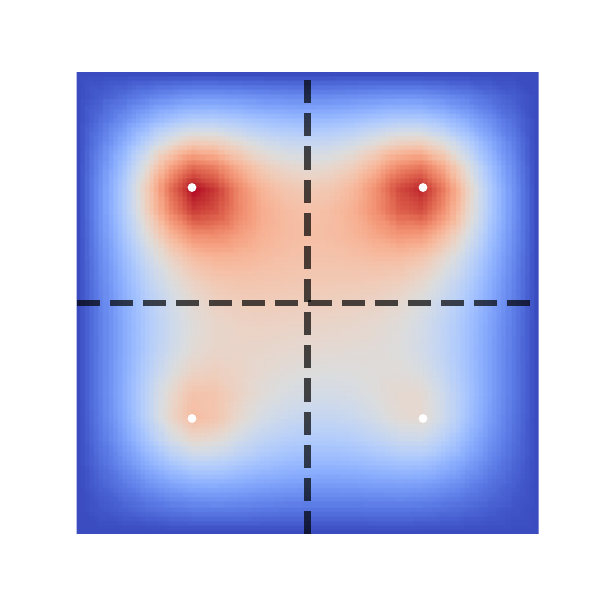}
    \includegraphics[width=0.3\linewidth]{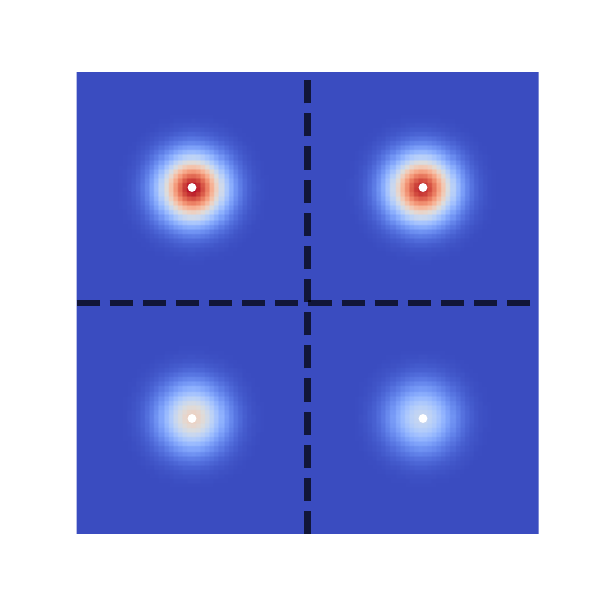}
    \caption{We show (left) the same solution $u$ of \cref{fig:poisson}. For $\ell_i(u)=\int u(x)e^{-\frac{(x-x_i)^2}{2\sigma^2}} \mathrm{d} x$ and $\sigma=0.1$ we see the projection $P_\Wm u$ of the solution $u$ onto the measurement space $\Wm$ for $V=H_0^1([-1,1]^2)$ (center) and $L^2([-1,1]^2)$ (right).}
    \label{fig:riesz_projection}
\end{figure}

Given that the observations can be seen as having been given an element $\omega\in \Wm$, we can view the task of state estimation as building a {\it recovery algorithm} which is defined as a mapping
\begin{equation*}
    A: \Wm \to V.
\end{equation*}
For such an algorithm $A$, the approximation of $u$ is given by
\begin{equation*}
    u^*=A(w)=A(P_\Wm u).
\end{equation*}
Note that, in our terminology, an algorithm $A$ is just a mapping from the observation space $\Wm$ to the ambient space $V$, and we do not attach any notion of practical feasibility to it. We will add this idea in a second stage.

The construction of $A$ should be based on the available prior
information that describes the properties of the unknown $u$, and the evaluation of its quality needs to be defined according to a precise criterion. By following similar lines as for forward reduced modeling, we distinguish two main quality benchmarks:
\begin{itemize}
    \item
    In the {\it worst case setting}, the sole prior information
    is that $u$ belongs to the solution manifold $\cM$ that we defined in equation \eqref{eq:manifold}. The performance of an algorithm $A$ over the class $\cM$ is measured by the ``worst case'' reconstruction error
    \begin{equation*}
        \label{eq:err-wc}
        \rE^\worstcase(\cM,A)\coloneqq\max_{u\in \cM} \norm{u-A(P_\Wm u)}.
    \end{equation*}
    In this case, the problem of finding an $A$ that minimizes $\rE^\worstcase(\cM, A)$ is
    called {\it optimal recovery}. It has been extensively studied for convex sets $\cM$ that are
    balls of smoothness classes \cite{Bojanov1994,MR1977,NW2008} but note that this is not the present case for our solution manifold.
    \item
    In the {\it average setting}, the prior information on $u$ is
    described by a probability distribution $p$ on $V$,
    which is supported on $\cM$. Typically, $p$ is induced by a probability distribution $\rho$ on $\Theta$ that is push-forwarded with the solution map: $p = u\#\rho$. It is then natural to measure the performance of an algorithm
    in an averaged sense, for example through the mean-square error
    \begin{equation*}
        \rE^\averagecase_\rho(\cM, A)\coloneqq\prt*{\int_\Theta \norm{u(\theta)-A(P_\Wm u(\theta))}^2 \rho(\d \theta)}^{1/2}.
        \label{eq:err-ms}
    \end{equation*}
    This average setting is the starting point for {\it Bayesian estimation}
    methods \cite{DS2017}. Let us observe that for any algorithm $A$ one has $\rE^\averagecase_\rho(\cM, A)(A)\leq \rE^\worstcase(\cM,A)$.
\end{itemize}

For both criteria, the best algorithm is the one that minimizes the reconstruction error. The optimal performance that we can achieve is thus
\begin{equation*}
    \delta^\star(\cM) \coloneqq \inf_{A:\Wm\to V} \rE^{\star}(\cM, A) , \quad \text{ with } \star = \{\averagecase, \worstcase\}.
\end{equation*}

For the worst case setting, a simple mathematical description of an optimal map that meets this benchmark was given in \cite{BCDDPW2017} (see also \cite{CDDFMN2020}). To define it, we note that in the absence of model bias and when a noiseless measurement $\omega=P_\Wm u$ is given, our knowledge on $u$ is that it belongs to the set
\begin{equation*}
    \cM_\omega:=\cM\cap (\omega+\Wm^\perp).
\end{equation*}
The set $\cM_\omega$ can be understood as the ``slice'' of the manifold $\cM$ which agrees with a given observation $\omega\in \Wm$. This ``slice'' can be a fully connected set, or composed of non connected sets, and it could even be the empty set depending on $\omega$. We refer to \cref{fig:bird} for a graphical illustration.
\begin{figure}[ht]
    \centering
    \includegraphics[scale=0.4]{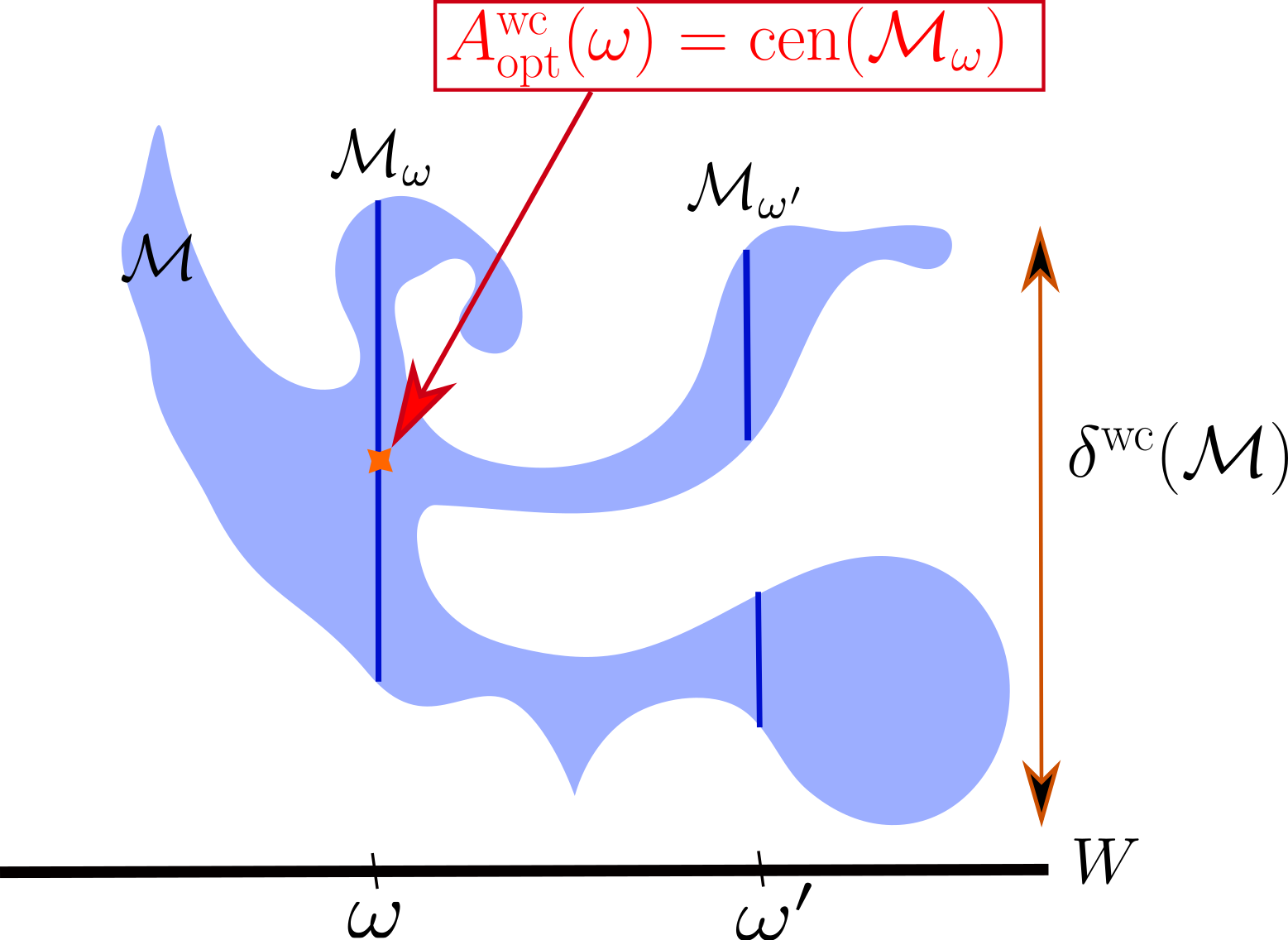}
    \caption{Illustration of the optimal recovery benchmark on a manifold in the two dimensional Euclidean space. Note that sometimes $\cM_\omega$ may be a non connected set as the figure depicts. This does not alter the fact that the best reconstruction is the center of the Chebyshev ball $\cen(\cM_\omega)$. Remark that the center does not necessarily lie in $\cM_\omega$ when the set is not connected.}
    \label{fig:bird}
\end{figure}

The best possible recovery map can be described through the following general notion.

\begin{definition}
    The Chebychev ball of a bounded set $S\in V$ is the closed ball $B(v,r)$ of minimal radius
    that contains $S$. One denotes by $v={\rm cen}(S)$ the Chebychev center of $S$ and
    $r={\rm rad}(S)$ its Chebychev radius.
\end{definition}

In particular
one has
\begin{equation}
    \frac 1 2 {\rm diam}(S)\leq {\rm rad}(S)\leq  {\rm diam}(S),
    \label{radiam}
\end{equation}
where ${\rm diam}(S):=\sup\{\norm{u-v}\cond u,v\in S\}$ is the diameter of $S$.
Therefore, the recovery map that minimizes the worst case error over $\cM_\omega$ for any given $\omega$, and therefore over $\cM$ is
defined by
\begin{equation}
    \label{eq:optimal-A-Cheb-center}
    A^{\worstcase}_{\opt}(\omega)={\rm cen}(\cM_\omega).
\end{equation}
Its associated worst case error is
\begin{equation*}
    \delta^{\worstcase}(\cM)= \sup \{\rad(\cM_\omega)\, :\, \omega\in \Wm\}.
\end{equation*}
Note that the map $A^{\worstcase}_{\opt}$ is also optimal among all algorithms for each slice $\cM_\omega$, where $\omega\in P_\Wm (\cM)$, since
\begin{equation*}
    \rE^{\worstcase}(A^{\worstcase}_{\opt},\cM_\omega)=\min_{A}\rE^{\worstcase}(A,\cM_\omega)=\rad(\cM_\omega), \quad \forall \omega\in P_\Wm(\cM).
\end{equation*}
However, there may exist other maps
$A$ such that $\rE^{\worstcase}(\cM,A)=\delta^{\worstcase}(\cM)$, since we
also supremize over $\omega\in P_\Wm(\cM)$.

In view of the equivalence \eqref{radiam}, we can
relate $\delta^{\worstcase}(\cM)$ to the quantity
\begin{equation*}
    \sigma\coloneqq\sup\{{\rm diam}(\cM_\omega)\,:\, \omega\in \Wm\}=\sup \{\norm{u-v}\; : \; u,v\in \cM, \;u-v\in \Wm^\bot \},
\end{equation*}
by the equivalence
\begin{equation*}
    \frac \sigma 2\leq \delta^{\worstcase}(\cM) \leq \sigma.
\end{equation*}

Note that having injectivity of the measurement map $P_\Wm: \cM \to \Wm$ is equivalent to $\sigma=0$. In other words, when the measurement map is rich enough so that there is only one element $u\in \cM$ that matches the observations, then the optimal algorithm will do a perfect recovery and $\delta^{\worstcase}(\cM)=0$.

More importantly, note that, in practice, the optimal map $A^{\worstcase}_{\opt}$ from \cref{eq:optimal-A-Cheb-center} cannot be easily constructed. Since the solution manifold $\cM$ is a high-dimensional and geometrically complex object, one cannot easily find the Chebyshev center to $\cM_\omega$ for a given measurement $\omega$. One is therefore interested in designing ``sub-optimal yet good'' recovery algorithms $A$ whose performance $E(\cM,A)$ does not deviate that much from the optimal $\delta^{\worstcase}(\cM)$. We next present a simple approach based on approximating $\cM$ with a linear approximation space $V_n$. We are going to see that using linear reduced model spaces $V_n$ as we did for forward reduced modeling (see \cref{sec:fwd:linear}) is a pretty good choice when the Kolomogorov width $d_n(\cM)$ decays fast, although this choice comes with a couple of caveats that we will point out.

\subsection{The Parametrized Background Data-Weak Algorithm (PBDW)}
\label{sec:linearPBDW}
Given a measurement space $\Wm$ and a reduced model $V_n$ with $1\leq n\leq m$, the PBDW algorithm is defined as the mapping
\begin{align}
    A_{m,n}:\Wm & \to V                                                                                               \\
    \omega      & \mapsto A_{m,n}(\omega) = \argmin_{v \in \omega + \Wm^\perp } \norm{v-P_{V_n}v} \label{eq:pbdw-lin}
\end{align}
The following theorem shows that the solution to the above minimization problem exists and is unique under some conditions that we will assume are satisfied in the rest of our developments. The result requires introducing the quantity
\begin{equation}
    \label{eq:infsup}
    \beta(F,H):=\inf_{f\in F}\sup_{h\in H}\frac {\inner{f,h}}{\norm{f}\, \norm{h}}=\inf_{f\in F}\frac {\norm{P_{H} f}}{\norm{f}} \in [0,1]
\end{equation}
where $(F,H)$ are any pair of finite-dimensional subspaces of $V$.
We also define the orthogonal projection into $F$ restricted to $H$,
\begin{align*}
    P_{F | H} : H & \to F                  \\
    h             & \mapsto P_{F | H} (h).
\end{align*}
For any $h\in H$, $P_{F | H}(h)$ is the unique element $f\in F$ such that
$$
\inner{h - f, \tilde f}= 0,\quad  \forall \tilde f \in F.
$$

\begin{theorem}
    \label{thm:explicit-sol-pbdw}
    Let $\Wm$ and $V_n$ be an observation space and a reduced basis such that $\beta(V_n, \Wm)>0$. Then the linear PBDW algorithm defined in \eqref{eq:pbdw-lin} is given by
    \begin{equation}
        \label{eq:explicit-u}
        A_{m,n}(\omega) = \omega + v^*_{m,n}(\omega) - P_\Wm\prt{v^*_{m,n}(\omega)} \; \in \widetilde{V}_n \coloneqq V_n\oplus(V_n^\perp\cap\Wm),
    \end{equation}
    with
    \begin{equation}
        \label{eq:explicit-v}
        v^*_{m,n}(\omega) = \prt{P_{V_n | \Wm} P_{\Wm | V_n}}^{-1} P_{V_n | \Wm} (\omega).
    \end{equation}
\end{theorem}
\begin{proof}
    By \cref{eq:pbdw-lin}, $A_{m,n}(\omega)$ is a minimizer of
    \begin{align*}
        \min_{u \in \omega + \Wm^\perp } \dist(u, V_n)^2
        & = \min_{u \in \omega + \Wm^\perp } \min_{v\in V_n} \Vert u - v\Vert^2                                                                \\
        & = \min_{v\in V_n}  \min_{\eta \in \Wm^\perp} \Vert \omega + \eta - v \Vert^2 \quad \text{($u=\omega+\eta$ with $\eta\in \Wm^\perp$)} \\
        & = \min_{v\in V_n} \Vert \omega -v - P_{\Wm^\perp}(\omega -v) \Vert^2                                                                 \\
        & = \min_{v\in V_n} \Vert \omega -v + P_{\Wm^\perp}(v) \Vert^2                                                                         \\
        & = \min_{v\in V_n} \Vert \omega - P_{\Wm}(v) \Vert^2. \label{eq:vn}
    \end{align*}
    The last minimization problem is a classical strongly convex, least squares optimization. Any minimizer $v^*_{m,n}=v^*_{m,n}(\omega)\in V_n$ satisfies the normal equations
    $$
    P^*_{\Wm|V_n}  P_{\Wm|V_n} v^*_{m,n} = P^*_{\Wm|V_n} \omega,
    $$
    where $P^*_{\Wm|V_n} : V_n \to \Wm$ is the adjoint operator of $P_{\Wm|V_n}$. Note that $P^*_{\Wm|V_n}$ is well-defined since $\beta(V_n, \Wm)= \min_{v\in V_n} \Vert P_{\Wm|V_n} v \Vert / \Vert v \Vert >0 $, which implies that $P_{\Wm|V_n}$ is injective and thus admits an adjoint. Furthermore, since for any $\omega \in \Wm$ and $v\in V_n$, $\inner{ v, \omega }=\inner{ P_{\Wm|V_n} v , \omega } = \inner{ v , P_{V_n|\Wm} \omega }$, it follows that $P^*_{\Wm|V_n} = P_{V_n|\Wm}$, which finally yields that the unique solution of the least squares problem is
    $$
    v^*_{m,n}(\omega) = \prt{P_{V_n | \Wm} P_{\Wm | V_n}}^{-1} P_{V_n | \Wm} (\omega) .
    $$
    Therefore $A_{m,n}(\omega) = \omega + \eta^*_{m,n}(\omega) = \omega + v^*_{m,n}(\omega) - P_\Wm v^*_{m,n}(\omega)$.
\end{proof}

The statement of \cref{thm:explicit-sol-pbdw} requires that $\beta(V_n,\Wm)>0$ because we need to guarantee injectivity of $P_{\Wm|V_n}$. For this to hold, the next Proposition shows that it is necessary (but not sufficient) that $n\leq m$.

\begin{proposition}
    \label{prop:n-m-condition}
    If $n> m$, then $\beta(V_n, \Wm)=0$.
\end{proposition}
\begin{proof}
    If $n>m$, then there exists an element $v\in V_n$ which is orthogonal to $\Wm$. Thus, for this element, $\norm{P_{\Wm}v}_V=0$ and therefore $\beta(V_n,\Wm)=0$.
\end{proof}

\subsection{Practical computation of $A_{m,n}(\omega)$ and of $\beta(V_n, \Wm)$}

\paragraph{Computation of $A_{m,n}(\omega)$:} Since $A_{m,n}(\omega)$ is given by \cref{eq:explicit-u}, computing it in practice boils down to computing $v^*_n(\omega)$. The explicit expression \eqref{eq:explicit-v} for $v^*_n(\omega)$ allows to easily derive its algebraic formulation. Let $F$ and $H$ be two finite-dimensional subspaces of $V$ of dimensions $n$ and $m$ respectively in the Hilbert space $V$ and let $\cF=\{f_i\}_{i=1}^n$ and $\cH=\{h_i\}_{i=1}^m$ be a basis for each subspace respectively. The Gram matrix associated to $\cF$ and $\cH$ is
$$
\tG(\cF, \cH) = \left(  \left< f_i, h_j\right> \right)_{\substack{1\leq i \leq n \\ 1\leq j \leq m}}.
$$
These matrices are useful to express the orthogonal projection
$P_{ F | H}: H\mapsto F$ in the bases $\cF$ and $\cH$ in terms of the matrix
\begin{equation}
    \label{eq:proj-matrix}
    \tP_{F | H} = \tG(\cF, \cF)^{-1} \tG(\cF, \cH).
\end{equation}
As a consequence, if $\cV_n = \{ v_i \}_{i=1}^n$ is a basis of the space $V_n$ and $\cW_m = \{\omega_i\}_{i=1}^m$ is the basis of $\Wm$ formed by the Riesz representers of the linear functionals $\{\ell_i\}_{i=1}^m$, the coefficients $\textbf{v}^*_{m,n}$ of the function $v^*_{m,n}$ in the basis $\cV_n$ are the solution to the normal equations
\begin{equation*}
    \label{eq:normal-eqs1}
    \tP_{V_n | \Wm} \tP_{\Wm | V_n}
    \textbf{v}^*_{m,n} =  \tP_{V_n | \Wm} \tG(\cW_m, \cW_m)^{-1} \textbf{z},
\end{equation*}
where $\textbf{z}=(\left< u, \omega_i\right>)_{i=1}^m=(z_i)_{i=1}^m$ is the vector of measurement observations that we introduced in \cref{eq:observations}. From \cref{eq:proj-matrix},
\begin{equation*}
    \begin{cases}
        \tP_{V_n | \Wm} & = \tG(\cV_n, \cV_n)^{-1} \tG(\cV_n, \cW_m), \\
        \tP_{\Wm | V_n} & = \tG(\cW_m, \cW_m)^{-1} \tG(\cW_m, \cV_n).
    \end{cases}
\end{equation*}
Usually $\textbf{v}^*_{m,n}$ is computed with a QR decomposition or any other suitable method for linear least-squares problems. Once $\textbf{v}^*_{m,n}$ is found, the vector of coefficients $\textbf{u}_{m,n}^*$ of $A_{m,n}(\omega)$ easily follows.

\paragraph{Computation of $\beta(V_n,\Wm)$:} As explained in the proof of \cref{thm:explicit-sol-pbdw}, if $n>m$, then $\beta(V_n,\Wm)=0$. We therefore focus on the case $n\leq m$, and remind that, from \cref{eq:infsup}, we have
\begin{equation*}
    \beta(V_n, \Wm) \coloneqq
    \min_{v\in V_n} \max_{\omega\in \Wm} \frac{\inner{v,w}}{\norm{v} \,\norm{w}}
    = \min_{v\in V_n} \frac{\Vert P_\Wm v \Vert}{\Vert v \Vert}.
\end{equation*}
The last equality comes from the fact that
$$
\max_{\omega\in \Wm} \frac{\inner{v,w}}{\norm{w}}
=  \max_{\omega\in \Wm} \frac{\inner{P_\Wm v,w}}{\norm{w}}
= \Vert P_\Wm v \Vert,
\quad \forall v \in V_n.
$$
Let $\cV_n = \{ v_i \}_{i=1}^n$ be a basis of the space $V_n$ and let $\textbf{c}$ be the coefficients of an element $v\in V_n$ in the basis $\cV_n$. For any nonzero $v\in V_n$, we can thus write
\begin{equation}
    \label{eq:beta-eigenvalue}
    \beta^2(V_n, \Wm) = \min_{v\in V_n}
    \frac{\Vert P_\Wm v \Vert_V^2}{\Vert v \Vert_V^2}
    =
    \min_{\textbf{c}\in \bR^n}
    \frac{\textbf{c}^T \tM(\cV_n, \cW_m) \textbf{c}}{\textbf{c}^T \tG(\cV_n, \cV_n) \textbf{c}}
\end{equation}
where
$$
\tM(\cV_n, \cW_m) \coloneqq \left( \left< P_\Wm v_i, P_\Wm v_j \right> \right)_{1\leq i, j \leq n}
$$
is a symmetric matrix.

Let us make a few remarks before giving an implementable expression for $\tM(\cV_n, \cW_m)$. First, we note that the value of $\beta(V_n, \Wm)$ does not depend on the selected bases $\cV_n$ and $\cW_m$. For example, using a basis $\widetilde \cV_n$ instead of $\cV_n$ amounts to changing the variable $\textbf{c}$ by $\widetilde{\textbf{c}} = \bU \textbf{c}$ for an invertible matrix $\bU$, and this does not affect the value of the minimizer. Second,  formula \eqref{eq:beta-eigenvalue} shows that $\beta(V_n, \Wm)$ is the smallest eigenvalue of the generalized eigenvalue problem
$$
\text{find } (\lambda, \textbf{c}) \in \bR\times \bR^n-\{0\} \quad \text{s.t.} \quad \tM(\cV_n, \cW_m) \textbf{c} = \lambda \tG(\cV_n, \cV_n) \textbf{c}.
$$
Since $\tG(\cV_n, \cV_n)$ and $\tM(\cV_n, \cW_m)$ are symmetric, positive definitive, the eigenvalues $\lambda$ are positive, and having $\beta(V_n, \Wm) >0$ is equivalent to the invertibility of $\tM(\cV_n, \cW_m)$. We can transform the generalized eigenvalue problem in a classical eigenvalue problem by multiplying by the inverse of $\tG(\cV_n, \cV_n)$. Also, remark that we have important simplifications when $\cV$ and/or $\cW_m$ are orthonomal bases since in that case $\tG(\cV_n, \cV_n) $ and $\tG(\cW_m, \cW_m)$ become the identity matrices.

We next give an explicit expression for $\tM(\cV_n, \cW_m)$. Since the coordinates in $\cV_n$ of the $i$-th basis function $v_i$ are given by the $i$-th canonical vector $\textbf{e}_i \in \bR^n$, using formula \eqref{eq:proj-matrix} we deduce that the coordinates of $P_\Wm v_i$ in $\cW_m$ are given by
$$
\textbf{p}_i \coloneqq \tP_{\Wm | V_n} \textbf{e}_i = \tG(\cW_m, \cW_m)^{-1} \tG(\cW_m, \cV_n) \textbf{e}_i,\quad \forall i \in \{1,\dots, n\}.
$$
Therefore
\begin{align*}
    \left< P_\Wm v_i, P_\Wm v_j \right>_V
    & = \textbf{p}_i^T \tG(\cW_m, \cW_m) \textbf{p}_j                                             \\
    & = \textbf{e}^T_i \tG^T(\cW_m, \cV_n)  \tG^{-1}(\cW_m, \cW_m) \tG(\cW_m, \cV_n) \textbf{e}_j
    ,\quad \forall (i,j) \in \{1,\dots, n\}^2,
\end{align*}
and
\begin{equation}
    \tM(\cV_n, \cW_m) = \tG^T(\cW_m, \cV_n)  \tG^{-1}(\cW_m, \cW_m) \tG(\cW_m, \cV_n).
    \label{eq:pbdw-correction}
\end{equation}

\begin{figure}[ht]
    \centering
    \includegraphics[scale=0.35]{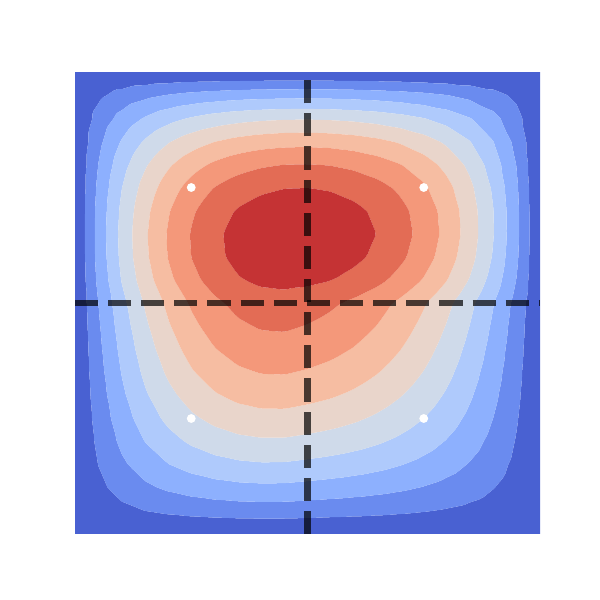}
    \includegraphics[scale=0.35]{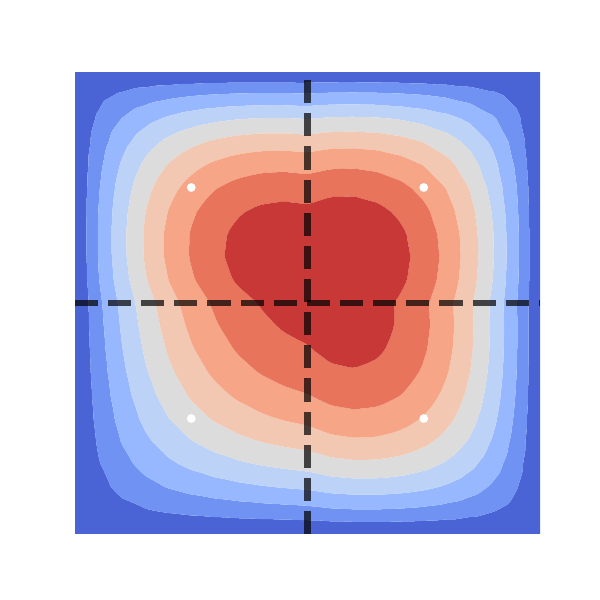}
    \includegraphics[scale=0.35]{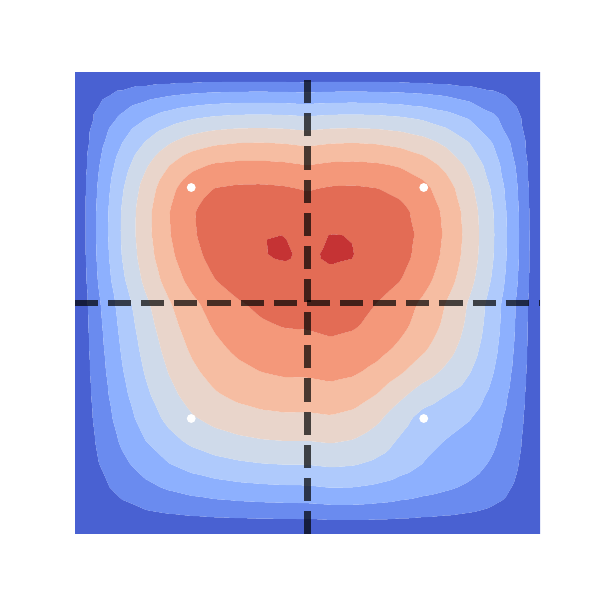}
    \caption{Following the same example as in \cref{fig:poisson} and \cref{fig:riesz_projection} we show here the true solution $u$ (left), $P_{V_n}u$ (center), that is, the best approximation of $u$ on $V_n$ for $n=2$ with $V_n$ obtained by an SVD strategy. At right the correction obtained by incorporating the riez representers information as in PBDW strategy (see \cref{eq:pbdw-correction}).}
    \label{fig:state-estimation}
\end{figure}

\subsection{Error bounds for PBDW}

Formula \eqref{eq:explicit-u} shows that $A_{m,n}$ is a bounded linear map from $\Wm$ to the space $\widetilde{V}_n$ which, as a reminder, was defined as $\widetilde{V}_n \coloneqq V_n\oplus(V_n^\perp\cap\Wm)$ in \eqref{eq:explicit-u}. We have the following bounds for the reconstruction error.

\begin{theorem}
    \label{thm:error-pbdw}
    If $\beta(V_n,\Wm)>0$, then
    \begin{equation}
        \label{eq:error-pbdw}
        \norm{u - P_{\widetilde{V}_n}u}\leq \norm{u - A_{m,n}(P_\Wm u)} \leq \beta^{-1}(V_n, \Wm) \norm{u - P_{\widetilde{V}_n}u}, \quad \forall u \in V.
    \end{equation}
\end{theorem}

\begin{proof}
    This result was first proven in \cite{BCDDPW2017}. Here we present a new alternative proof which significantly simplies the original one.
    To show the lower bound, it suffices to notice that since $A_{m,n}(P_\Wm u)\in \widetilde{V}_n$, then
    $$
    \norm{u - A_{m,n}(P_\Wm u)}
    \geq \min_{\tilde v\in \widetilde{V}_n} \norm{u-v}
    = \norm{u - P_{\widetilde{V}_n}u}.
    $$
    The proof of the upper bound crucially relies on the fact that
    \begin{equation*}
        \beta(V_n, \Wm) = \beta(\widetilde{V}_n, \Wm) = \beta(\Wm^\perp, \widetilde{V}_n^\perp),
    \end{equation*}
    which is a statement that we prove in \cref{lem:beta-technical}. Since, by construction, $P_{\Wm}\prt*{A_{m,n}(P_\Wm u)}= P_{\Wm}(u)$, then $u-A_{m,n}(P_\Wm u) \in \Wm^\perp$, and from the definition of $\beta(\Wm^\perp, \widetilde{V}_n^\perp)$ it holds that
    $$
    \beta(\Wm^\perp, \widetilde{V}_n^\perp)
    \norm{u-A_{m,n}(P_\Wm u)}
    \leq \sup_{\tilde v^\perp \in \widetilde{V}_n^\perp}
    \dfrac{\inner{u-A_{m,n}(P_\Wm u), \tilde v^\perp}}{\norm{\tilde v^\perp}}.
    $$
    In addition, since $A_{m,n}(P_\Wm u)\in \widetilde{V}_n$, by orthogonality we have
    $$
    \inner{u-A_{m,n}(P_\Wm u), \tilde v^\perp}
    =\inner{u-\tilde v, \tilde v^\perp}, \quad \forall (\tilde v, \tilde v^\perp) \in \widetilde{V}_n\times \widetilde{V}_n^\perp
    $$
    Consequently,
    $$
    \beta(\Wm^\perp, \widetilde{V}_n^\perp)
    \norm{u-A_{m,n}(P_\Wm u)}
    \leq \sup_{\tilde v^\perp \in \widetilde{V}_n^\perp}
    \dfrac{\inner{u-\tilde v, \tilde v^\perp}}{\norm{\tilde v^\perp}}=\norm{u-\tilde v},\quad \forall \tilde v \in \widetilde{V}_n,
    $$
    Therefore, dividing by $\beta(\Wm^\perp, \widetilde{V}_n^\perp)=\beta(V_n, \Wm)>0$ \corr{(we prove this equality in \cref{lem:beta-technical})}, and minimizing over $\tilde v\in \widetilde{V}_n$ we obtain the desired bound,
    $$
    \norm{u-A_{m,n}(P_\Wm u)}
    \leq \frac{1}{\beta(V_n, \Wm)} \norm{u- P_{\widetilde{V}_n}u}.
    $$
\end{proof}

\begin{lemma}
    \label{lem:beta-technical}
    Suppose that $\beta(V_n, \Wm)>0$. Then it holds that
    \begin{equation*}
        \beta(V_n, \Wm) = \beta(\widetilde{V}_n, \Wm) = \beta(\Wm^\perp, \widetilde{V}_n^\perp).
    \end{equation*}
\end{lemma}

\begin{proof}
    Since we assume $\beta(V_n, \Wm)>0$, necessarily $n\leq m$ by \cref{prop:n-m-condition}. Now, to see that $\beta(V_n, \Wm) = \beta(\widetilde{V}_n, \Wm)$, we consider an element $\tilde v \in V_n^\perp \cap \Wm \subset \widetilde{V}_n$. Since $\tilde v\in \Wm$, $\norm{P_{\Wm}\tilde v}/\norm{\tilde v}=1$. Therefore, since $\beta(\widetilde{V}_n, \Wm)\leq 1$,
    $$
    \beta(\widetilde{V}_n, \Wm)
    = \min_{\tilde v\in \widetilde{V}_n} \frac{\norm{P_{\Wm}(\tilde v)}}{\norm{\tilde v}}
    = \min_{v\in V_n} \frac{\norm{P_{\Wm} v}}{\norm{v}} = \beta(V_n,\Wm).
    $$
    The equality $\beta(\widetilde{V}_n, \Wm) = \beta(\Wm^\perp, \widetilde{V}_n^\perp)$ is a direct application of \cite[Proposition A.1]{MMPY2015}.
\end{proof}

As a direct consequence of \cref{thm:error-pbdw}, we can derive error bounds on the reconstruction quality of $A_{m,n}$ on the whole set $\cM$. This is recorded in the next Corollary.

\begin{corollary}
    It holds that
    \begin{align}
        \cE^{\worstcase}\prt{\cM, \widetilde{V}_n}\leq \rE^\worstcase(\cM, A_{n,m})               & \leq \beta^{-1}(V_n, \Wm) \cE^{\worstcase}\prt{\cM, \widetilde{V}_n}        \\
        \cE^{\averagecase}_\rho\prt{\cM, \widetilde{V}_n}\leq \rE^\averagecase_\rho(\cM, A_{n,m}) & \leq \beta^{-1}(V_n, \Wm) \cE^{\averagecase}_\rho\prt{\cM, \widetilde{V}_n}
    \end{align}
\end{corollary}

\begin{proof}
    It suffices to maximize \eqref{eq:error-pbdw} over $u\in \cM$ or take the expectation to obtain the inequalities.
\end{proof}

\subsection{Interpretation of $\beta(V_n, \Wm)$ and other remarks}

The lower bound \eqref{eq:error-pbdw} implies that $A_{m,n}(\omega)$ provides the best reconstruction in $\widetilde{V}_n$ that we can expect as the reconstruction error is given by the best approximation of $u$ in $\widetilde{V}_n$, which is the orthogonal projection of $u$ into $\widetilde{V}_n$.

The upper bound of \eqref{eq:error-pbdw} shows that we deviate from the best approximation up to the multiplicative constant $\beta^{-1}(V_n, \Wm)$. For this reason, we can view $\beta(V_n, \Wm)$ as a stability factor. This quantity can be interpreted as the cosine of the angle between $V_n$ and $\Wm$ (see \cref{fig:angle}),
$$
\beta(V_n, \Wm) = \cos(\theta_{V_n, \Wm}).
$$
The stability of the algorithm thus depends on an interplay between the approximation space $V_n$ and the observation space $\Wm$. Of course, if $\Wm$ is given, by choosing the reduced space $V_n=\Wm$, we would maximize the stability constant, and we would have $\beta(V_n, \Wm)=1$. This is the most stable situation, but this choice also affects the quality of approximation: in general, the observation space $\Wm$ has rather poor approximation properties so by choosing $V_n=\Wm$, we degrade $\cE^{\worstcase}\prt*{\cM, \widetilde{V}_n}$. Alternatively, if the application allows for it, we can work with a reduced model $V_n$ that gives a good accuracy (e.g.~the one from forward reduced order modeling), and we can consider the problem of optimizing the choice of the measurement space $\Wm$. This would be done by searching for optimal $\ell_i$ (respectively $\omega_i$) from a given dictionary of admissible linear functionals. We do not cover the topic of optimal sensor placement in this tutorial but we refer to \cite{MMT2016, BCMN2018} for works on this topic, and to \cite{Mula2022} for a tutorial covering this aspect.

Another interesting observation is that for a fixed measurement space $\Wm$ we still have the freedom to tune the dimension $n$ of $V_n$ to improve accuracy. Consequently, we can search for
$$
n^* = \argmin_{1\leq n\leq m}\; \rE^\star(A_{n,m}, \cM),\quad \text{with }\star=\{\worstcase, \averagecase\}
$$
and work with $A_{m,n^*}$. Contrary to what one might expect at first glance, the best performance is not achieved in general when we take $n$ as large as possible, namely $n=m$. In fact, there is a sweet spot due to the trade-off between the increase of the approximation properties of $V_n$ as $n$ grows, and the degradation of the stability of the algorithm, given here by the decrease of $\beta(V_n, \Wm)$ to 0 as $n\to m$.

We finish this section by connecting the algorithm $A_{m,n}$ with the concept of optimal algorithms in the sense of $\delta^\worstcase(\cM)$. For this, suppose that the reduced model $V_n$ is such that we approximate the elements of $\cM$ at $\eps_n>0$ accuracy,
$$
\cE^\worstcase(\cM; V_n) \leq \eps_n.
$$
As a consequence, $\cM$ is included in the cylinder (see \cref{fig:Kn})
\begin{equation}
    \label{eq:cylinder}
    \cM \subset \cK_n\coloneqq \{ v\in V \; : \: {\rm dist}(v,V_n)\leq \eps_n\}.
\end{equation}
We can prove (see \cite{CDDFMN2020}) that $A_{m,n}$ is optimal among all linear and nonlinear algorithms when $\cM$ is replaced by the simpler containment set $\cK_n$, that is
$$
A_{m,n} = \argmin_{A:\Wm\to V}\; \rE^{\worstcase}(A, \cK_n).
$$

\begin{figure}
    \centering
    \begin{subfigure}[t]{0.4\textwidth}
        \centering
        \includegraphics[width=\textwidth]{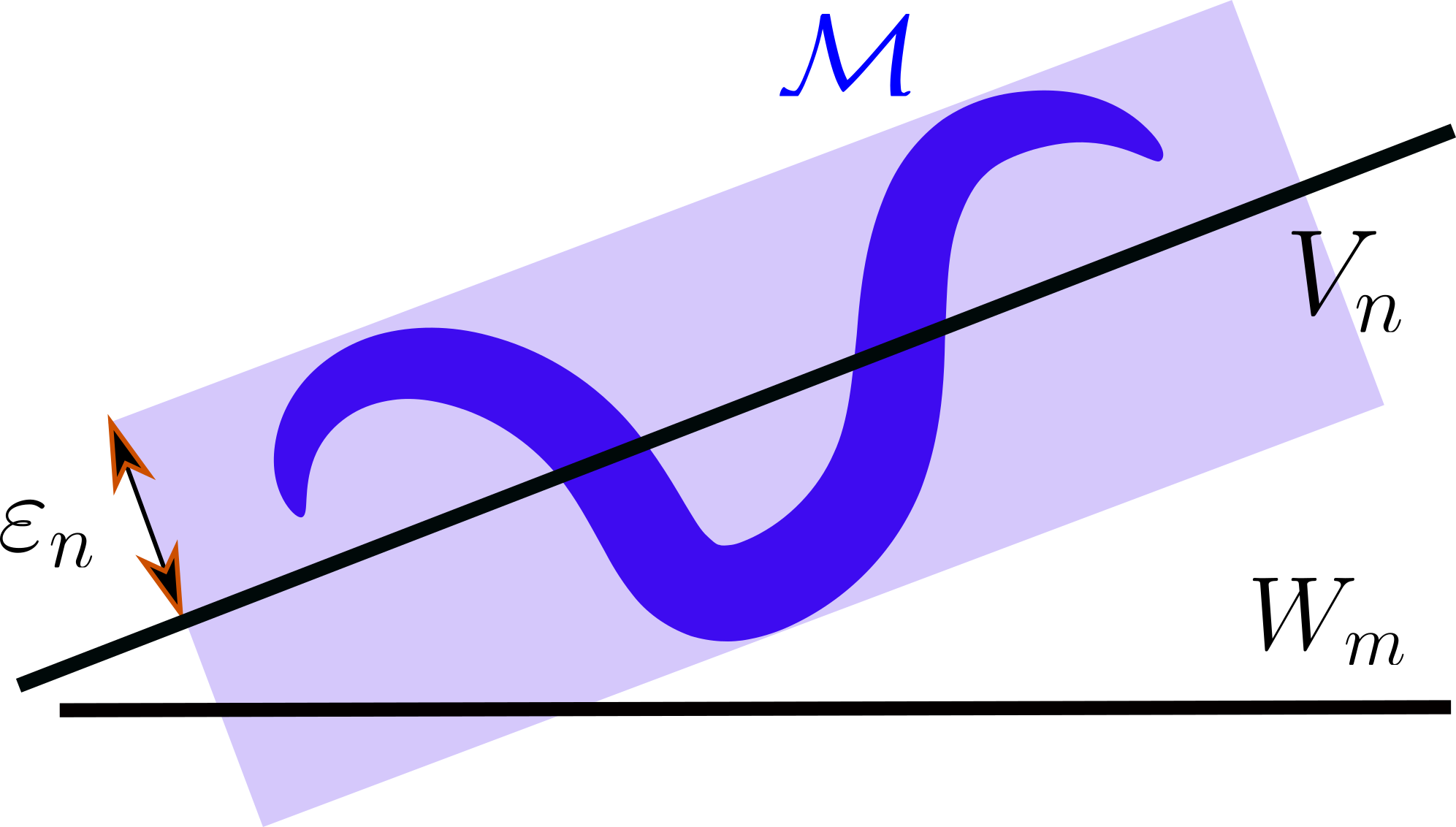}
        \caption{A manifold $\cM$, a linear space $V_n$ with accuracy $\eps_n$, and the cylinder $\cK_n$ from \cref{eq:cylinder}.}
        \label{fig:Kn}
    \end{subfigure}
    \begin{subfigure}[t]{0.4\textwidth}
        \centering
        \includegraphics[width=0.8\textwidth]{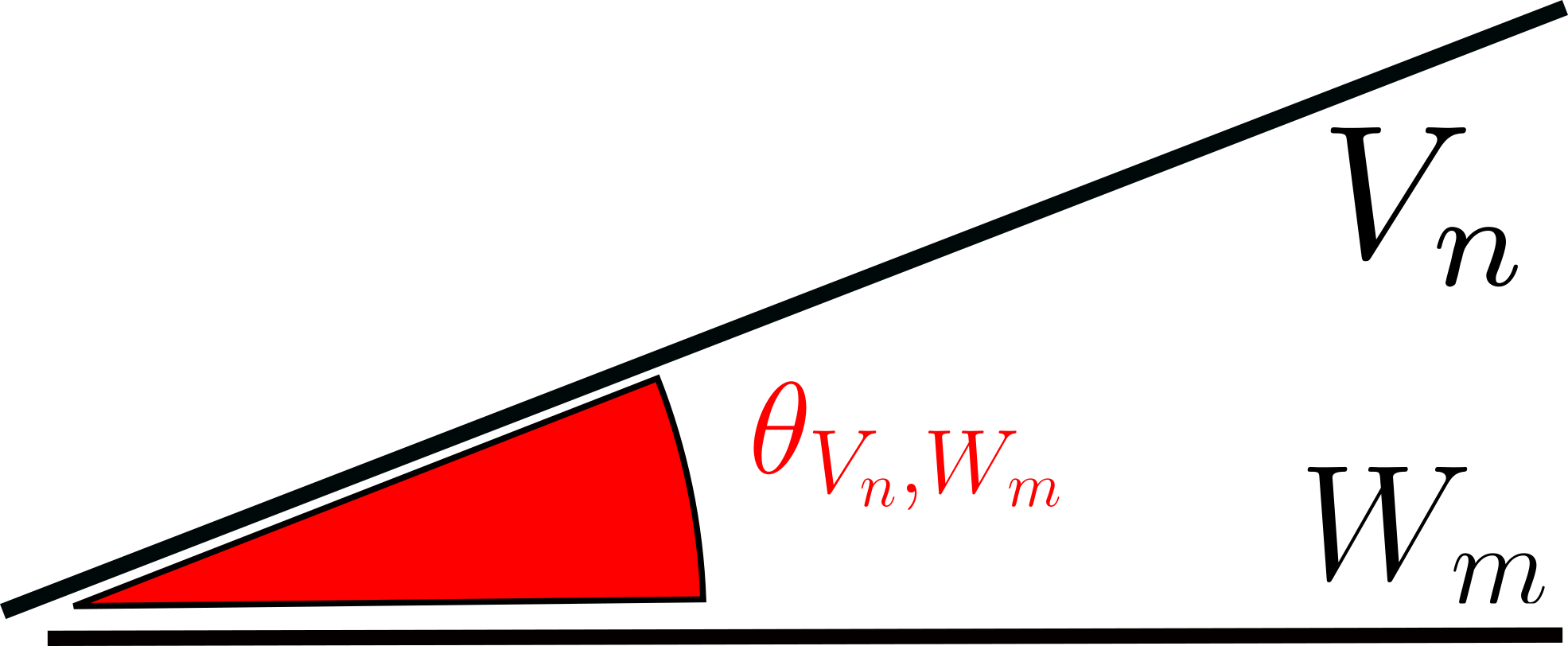}
        \caption{Angle between $V_n$ and $\Wm$.}
        \label{fig:angle}
    \end{subfigure}
    \caption{The concepts associated to the linear reconstruction algorithm.}
    \label{fig:one-space}
\end{figure}

\subsection{Further Reading and Applications}
This document covers quite extensively how to work with linear approximation subspaces $V_n$. Extending this setting to nonlinear spaces $V_n$ is a very dynamic research direction for which several milestones have already been achieved. First, the setting can easily be extended to affine approximation spaces of the form $\bar u+V_n$, as is explained in \cite{CDDFMN2020}. Next, one can take affine spaces as a building block to construct piece-wise affine approximations. The idea is to build a partition of the manifold $\cM$ into $K$ disjoint subsets, $\cM=\cup_{k=1}^K \cM_k$ with $\cM_k \cap \cM_{k'}=\emptyset$ if $k\neq k'$. We approximate the elements of each subset $\cM_k$ with affine subspaces $\bar u_k+V_n^{(k)}$. To do state estimation, given the observation $\omega$, we need to find a procedure to detect in which subset $\cM_k$ lies the target function $u$, and then we reconstruct with the affine version of PBDW. A strategy to implement this idea was proposed in \cite{CDMN2022}, and it was also shown in that if we are willing to have enough partitions, then the piece-wise affine strategy is near-optimal in the sense that its performance will get close to $\delta^{\worstcase}(\cM)$. Additionally, we refer to \cite{NP2024} for another construction based on randomization. Another relevant extension is the one introduced in \cite{MPV2025} where the linear approximation paradigm is made time-dependent with dynamical approximation techniques which can, in addition, preserve certain structures such as symplecticity. A first contribution involving fully-nonlinear approximation families was proposed in \cite{AABGMT2025}.

\new{Last but not least, we mention that the above-discussed paradigm has found numerous applications including biomedical problems \cite{TPYP2018, GGLM2021, GLM2021, GLM2022, galarce2023displacement, galarce2025bias, mantegazza2026non}, food freezing \cite{GRPCC2025},  electromagnetism \cite{alahyane2026optimal}, pollution dispersion \cite{HCBM2019, DMS2024} and nuclear engineering \cite{ABCGMM2018, TLMMT2023, bao2025spatiotemporal, riva2025data, riva2026real}. The idea of PBDW and the stability concept connected to the parameter $\beta(\Vn, \Wm)$ has also recently been used to develop dynamical sampling strategies to compute integrals arising in the framework of solving forward problems with fully nonlinear parametric approximations of the type introduced in \eqref{eq:decoder}.}

    \begin{multicols}{2}
  \begin{acronym}[WLOG] \itemsep = 0.2em
    \acro{gd}[GD]{gradient descent}
    \acro{lhs}[LHS]{left-hand side}
    \acro{mds}[MDS]{multi-dimensional scaling}
    \acro{pde}[PDE]{Partial Differential Equation}
    \acro{ppde}[pPDE]{parametrized Partial Differential Equation}
    \acro{relu}[ReLU]{rectified linear unit}
    \acro{rhs}[RHS]{right-hand side}
    \acro{rom}[ROM]{Reduced Order Modeling}
    \acro{sgd}[SGD]{stochastic gradient descent}
    \acro{silu}[SiLU]{sigmoid-weighted linear unit}
    \acro{wlog}[WLOG]{without loss of generality}
    \acro{wsn}[WSN]{wireless sensor network}
  \end{acronym}
\end{multicols}

    \section*{Bibliography}
    \addcontentsline{toc}{section}{Bibliography}

    \begingroup
    \bibliography{./content/references.bib}
    \endgroup
\end{document}